\documentclass[11pt,a4paper]{article}

\usepackage[margin=2cm]{geometry}

\usepackage{amsmath,amssymb,amsthm}
\usepackage[numbers]{natbib}

\usepackage[sans]{dsfont}
\newcommand{\unito}{{\mathds 1}}

\let\<\langle
\let\>\rangle

\DeclareMathOperator\curry  {\Lambda}
\DeclareMathOperator\uncurry{\Lambda^{-}}

\DeclareMathOperator\ev{ev}
\DeclareMathOperator\id{id}
\DeclareMathOperator\Id{Id}

\DeclareMathOperator\C{\mathbf{C}}
\DeclareMathOperator\R{\mathbb{R}}

\usepackage[all]{xy}
\SelectTips{cm}{10}

\IfFileExists{url.sty}{\usepackage{url}}{}
\providecommand{\url}[1]{{\tt #1}}

\usepackage{ifpdf}
\ifpdf
    \IfFileExists{hyperref.sty}{\usepackage[pdftex]{hyperref}}{}
\else
    \IfFileExists{hyperref.sty}{\usepackage[hypertex]{hyperref}}{}
\fi

\swapnumbers

\newtheorem{lemma}[subsubsection]{Lemma}
\newtheorem{theorem}[subsubsection]{Theorem}
\newtheorem{corollary}[subsubsection]{Corollary}
\newtheorem{proposition}[subsubsection]{Proposition}

\theoremstyle{definition}
\newtheorem{remark}[subsubsection]{Remark}
\newtheorem{example}[subsubsection]{Example}
\newtheorem{definition}[subsubsection]{Definition}

\numberwithin{equation}{subsection}

\newcommand{\lemref}[1]{Lemma~\ref{#1}}

\newcommand{\corref}[1]{Corollary~\ref{#1}}
\newcommand{\proref}[1]{Proposition~\ref{#1}}

\newcommand{\remref}[1]{Remark~\ref{#1}}

\title{Tangent bundles in differential $\lambda$-categories}
\author{Oleksandr Manzyuk}

\begin{document}
\maketitle

\begin{abstract}
  Differential $\lambda$-categories were introduced by \citet{1857313}
  as models for the simply typed version of the differential
  $\lambda$-calculus of \citet{MR2016523}.  A differential
  $\lambda$-category is a cartesian closed differential category of
  \citet{MR2591951} in which the differential operator is compatible
  with the closed structure.  We prove that any differential
  $\lambda$-category is equipped with a canonical strong commutative
  monad whose construction resembles that of the tangent bundle in the
  category of smooth manifolds.  Most of the results of this note
  remain valid in an arbitrary cartesian differential category.  Our
  emphasis on differential $\lambda$-categories is motivated by the
  anti\-cipated application of the theory developed in this note to
  the design and semantics of a $\lambda$-calculus extended by the
  pushforward operator.
\end{abstract}

\section{Introduction}

Differential $\lambda$-categories were introduced by \citet{1857313}
as models for the simply typed version of the
differential $\lambda$-calculus of
\citet{MR2016523}.  The notion of differential $\lambda$-category is a
refinement of the notion of cartesian differential category introduced
by \citet{MR2591951}.  \citeauthor{MR2016523} drew the
motivation for extending the $\lambda$-calculus with differential
operators from linear logic.  However, the differential
$\lambda$-calculus is also an attractive foundation on which to build
a functional programming language with built-in support for
differentiation.  Unlike, for example, symbolic differentiation, the
differential $\lambda$-calculus can handle not only mathematical
expressions, but arbitrary $\lambda$-terms.  Most notably, it can
take derivatives through and of higher-order functions.  Like symbolic
differentiation, the differential $\lambda$-calculus, implemented
naively, yields a grossly inefficient way to compute derivatives,
suffering from the loss of sharing.
The purpose of this note is to extend the semantic theory of
differential $\lambda$-categories so as to be able to build on top
of it a variation on the differential $\lambda$-calculus which
would not necessitate this loss of efficiency.

Automatic differentiation (AD) is a technique for efficiently
computing derivatives.  There are several variations of AD.  The
easiest to explain is so called \emph{forward mode} AD, which is based
on the following ideas from differential geometry.  For a smooth
manifold $X$, denote by $TX$ its tangent bundle.  For example, the
tangent bundle of $\R^n$ can be identified with $\R^n\times \R^n$, the
space of pairs $(x', x)$ consisting of a point $x\in\R^n$ and a
tangent vector $x'\in\R^n$ at that point.  For a smooth map between
smooth manifolds $f : X\to Y$, denote by $Tf : TX\to TY$ the
pushforward of $f$.  For a smooth map $f: \R^m\to\R^n$, the
pushforward $Tf$ is given by $Tf (x', x) = (J_f(x)\cdot x', f(x))$,
where $J_f(x)$ is the Jacobian of $f$ at the point $x$.  The
correspondences $X\mapsto TX$, $f\mapsto Tf$ constitute a functor from
the category of smooth manifolds to itself; the functoriality of $T$
reduces to the chain rule for derivatives.  In particular, if $f$ is
the composition of $f_1$, $f_2$, \dots, $f_k$, then $Tf$ is the
composition of $Tf_1$, $Tf_2$, \dots, $Tf_k$.  Furthermore, the
functor $T$ preserves products.  Therefore, in order to compute the
pushforward of a compound function it suffices to know the
pushforwards of its constituents, which is what various
implementations of forward mode AD take advantage of.  One remarkable
property of forward mode AD is the following complexity guarantee:
evaluation of the pushforward takes no more than a constant factor
times as many operations as evaluation of the function.

Although AD of first-order programs is well understood (e.g., see the
textbook by \citet{MR1753583}), surprisingly little is known about AD
in the presence of first-class functions.  Handling of higher-order
functions becomes a delicate issue.  \citet{1466794} discuss some
problems arising when one tries to extend a functional language with
AD operators.  They describe \citep{Siskind2008UPU} a novel AD system,
\textsc{Stalin}$\nabla$, and claim that it correctly handles
higher-order functions; unfortunately, no proof of that claim is
given, and, in fact, no formal theory supporting it is developed.  We
hope that bridging a gap between the differential $\lambda$-calculus
and forward mode AD will shed some light on these problems.  By doing
so, we expect to lay down a solid theoretical foundation for an
efficient implementation of a functional programming language with
built-in support for differentiation.  Ultimately, we hope to design a
$\lambda$-calculus that is similar to the differential
$\lambda$-calculus but is built around the idea of pushforward instead
of derivative.

In this note we introduce an analogue of the tangent bundle functor in
any differential $\lambda$-category.  Differential
$\lambda$-categories are models for the simply typed version of the
differential $\lambda$-calculus, and we hope that the pushforward
construction can be captured as a syntactic operation in a
$\lambda$-calculus.  We study the properties of the tangent bundle
functor and show that it is part of a strong commutative monad.  This
is an encouraging result because it establishes a link with
computational $\lambda$-calculi of \citet{MR1115262}.

The tangent bundle functor on a cartesian differential category is
also considered by \citet{CockettCruttwell}.  They show that it is an
example of abstract ``tangent structure'', which is an axiomatization
of differential structure at the level of smooth manifolds.
\citeauthor{CockettCruttwell} prove that any tangent structure has the
structure of a monad, thus partly replicating the results of this
note.

For readers familiar with synthetic differential geometry
\citep{MR2244115}, it should not come as a surprise that the tangent
bundle functor is part of a strong commutative monad.  Indeed,
synthetic differential geometry is developed relative to a topos that
is assumed to contain an object $D$ of ``infinitesimals''.  The
tangent bundle of a space $X$ is then defined as the exponential
$X^D$, and it is a general fact having nothing to do with
differentiation that the functor $(-)^D$ can be equipped with the
structure of a strong commutative monad.  For example, the
multiplication is defined as the composite
\[
(X^D)^D \simeq X^{D\times D} \xrightarrow{X^\Delta}X^D,
\]
where $\Delta: D\to D\times D$ is the diagonal morphism.  However, we
still think that the results presented in this note are of some
interest, as they operate at a more basic level.  Topoi enjoy many
powerful properties.  In contrast, the notion of tangent bundle and
its associated algebraic structures make sense in any cartesian
differential category, which is a minimal setup in which differential
calculus and other notions reminiscent of it can be described.
Differentiation appears in different guises in both combinatorics and
computer science.  For example, \citet{1857313} provide two examples
of differential $\lambda$-categories of combinatorial rather than
analytic nature.  Consequently, the results of this note apply to
these categories.  On the other hand, it is not clear if these
categories can be embedded into larger categories in such a way that
the tangent bundle functor becomes representable, so that the ideas
from synthetic differential geometry can be applied.

It is rather amusing and instructive to see how the algebraic
structures existing on the tangent bundle functor can be derived
directly from the general properties of differentiation.  That is why
we have chosen to present the proofs with full details.  Although we
hope that the results of this note will be of independent interest to
category theorists, our primary motivation for developing the theory
of tangent bundles in differential $\lambda$-categories is the desire
to formulate forward mode AD in the form of a $\lambda$-calculus
extended by AD operators.  We consider this in a sequel to this note.

\paragraph{Acknowledgments.}  I would like to thank Alexey Radul for
lots of fruitful discussions, as well as for carefully reading this
manuscript and making suggestions that have improved the exposition.

\section{Preliminaries}

We begin by summarizing the key concepts.  We follow the notation of
\citet{1857313}.

\subsection{Cartesian categories}

Let $\C$ be a cartesian category and $X$, $Y$, $Z$ arbitrary objects
of $\C$.  We denote by $X\times Y$ the product of $X$ and $Y$ and by
$\pi_1 : X\times Y\to X$, $\pi_2 : X\times Y\to Y$ the projections.
The terminal object is denoted by $\unito$, and for any object $X$, we
denote by $!_X$ the unique morphism from $X$ to $\unito$.  For a pair
of morphisms $f : Z\to X$ and $g : Z\to Y$, denote by $\< f, g\> :
Z\to X\times Y$ the pairing of $f$ and $g$, i.e., the unique morphism
such that
\begin{equation}
  \label{eq:proj}
  \pi_1\circ\< f, g\> = f
  \quad \textup{and} \quad
  \pi_2\circ\< f, g\> = g.
\end{equation}
The following equations follow immediately from the universal property
of pairing:
\begin{align}
\<\pi_1, \pi_2\> &= \id,
\label{eq:id-pair}
\\
\<f, g\> \circ h &= \<f \circ h, g \circ h\>.
\label{eq:circ-pair}
\end{align}
For a pair of morphisms $f: X\to Y$ and $g: U\to V$, denote by
$f\times g : X\times U\to Y\times V$ the product of $f$ and $g$, i.e.,
the unique morphism such that
\begin{equation}
  \label{eq:times-def}
  \pi_1\circ (f\times g) = f\circ\pi_1
  \quad \textup{and} \quad
  \pi_2\circ (f\times g) = g\circ\pi_2.
\end{equation}
Comparing these equations with \eqref{eq:proj}, we conclude that
\begin{equation}
  \label{eq:times-pair}
  f\times g = \<f\circ\pi_1, g\circ\pi_2\>.
\end{equation}
Equations \eqref{eq:proj}, \eqref{eq:circ-pair}, and
\eqref{eq:times-pair} imply that for any morphisms $f: X\to U$, $g:
Y\to V$, $h: Z\to X$, and $k: Z\to Y$ holds
\begin{equation}
  \label{eq:times-circ-pair}
  (f\times g) \circ \<h, k\> = \<f\circ h, g\circ k\>.
\end{equation}
Any cartesian category is a symmetric monoidal category with the
tensor product and unit object given by $\times$ and $\unito$,
respectively.  The associativity constraint $a_{X,Y,Z}: (X\times
Y)\times Z\to X\times (Y\times Z)$ is given by $a_{X,Y,Z} =
\<\pi_1\circ\pi_1, \<\pi_2\circ\pi_1, \pi_2\>\>$.  The left and right
unit constraints are given by $\ell_X=\<!_X,\id_X\>: X\to\unito\times
X$ and $r_X=\<\id_X, !_X\>:X\to X\times\unito$, respectively.  The
symmetry $c_{X,Y}:X\times Y\to Y\times X$ is given by
$c_{X,Y}=\<\pi_2,\pi_1\>$.  The following equations are
straightforward:
\begin{align}
  c\circ \<f, g\> &= \<g, f\>\label{eq:sym-pair}\\
  c\circ (h\times k) &= (k\times h)\circ c.\label{eq:sym-times}
\end{align}
Out of the associativity and commutativity isomorphisms one can
construct the distributivity isomorphism $\sigma: (A\times B)\times
(C\times D)\to (A\times C)\times (B\times D)$.  Explicitly, it is
given by $\sigma=\<\<\pi_1\circ\pi_1, \pi_1\circ\pi_2\>,\linebreak[3]
\<\pi_2\circ\pi_1, \pi_2\circ\pi_2\>\>$.  Using
equations~\eqref{eq:proj}--\eqref{eq:times-circ-pair} one can easily
prove the following equations:
\begin{align}
  \label{eq:sigma-pair}
  \sigma \circ \<\<f, g\>, \<h,k\>\> = \<\<f, h\>, \<g, k\>\>,
  \\
  \label{eq:sigma-times}
  \sigma \circ (\<f,g\>\times \<h, k\>) = \<f\times h, g\times k\>.
\end{align}
We always define morphisms in a cartesian category rigorously, using
the combinators $\<-,-\>$ and $\times$, giving the preference to the
former.  The proofs of the equations involving the morphisms so
defined rely on the properties of the combinators $\<-,-\>$ and
$\times$ stated above.  Sometimes, however, this approach can lead to
rather obscure definitions.  We then also write, for illustration
purposes, the morphism being defined using set-theoretic notation,
pretending that our underlying category is a category of sets with
some structure.  For example, set-theoretically, the distributivity
isomorphism $\sigma$ is given by $\sigma((a, b), (c, d)) = ((a, c),
(b, d))$.

\subsection{Cartesian closed categories}

A cartesian category $\C$ is called \emph{closed} if for any pair of
objects $X$ and $Y$ of $\C$ there exist an object $X\Rightarrow Y$,
called the \emph{exponential object}, and a morphism $\ev_{X,Y}:
(X\Rightarrow Y)\times X\to Y$, called the \emph{evaluation morphism},
satisfying the following universal property: the map $\uncurry : \C(Z,
X\Rightarrow Y) \to \C(Z\times X, Y)$ given by $\uncurry(g) =
\ev_{X,Y}\circ (g\times \id_X)$ is bijective.  Let $\curry :
\C(Z\times X, Y) \to \C(Z, X\Rightarrow Y)$ denote the inverse of
$\uncurry$; i.e., for a morphism $f : Z\times X\to Y$, $\curry(f) : Z
\to X\Rightarrow Y$ is the unique morphism such that $\ev_{X, Y}\circ
(\curry(f)\times \id_X) = f$.  The morphism $\curry(f)$ is called the
\emph{currying} of $f$.  We shall frequently use the equation
\begin{align}
  \curry(f)\circ g & = \curry(f\circ(g\times\id)),
  \label{eq:curry}
\end{align}
which follows immediately from the definition of $\curry$.

\subsection{Cartesian differential categories}

The notion of cartesian differential category was introduced by
\citet{MR2591951} as an axiomatization of differentiable maps as well
as a unifying framework in which to study different notions
reminiscent of the differential calculus.

\begin{definition}[{\citep[Definition~1.1.1]{MR2591951}}]
  A category $\C$ is \emph{left-additive} if each homset is equipped
  with the structure of a commutative monoid $(\C(X, Y), +, 0)$ such
  that $(g + h) \circ f = (g\circ f) + (h\circ f)$ and $0\circ f = 0$.
  A morphism $f$ in $\C$ is \emph{additive} if it satisfies $f\circ (g
  + h) = (f\circ g) + (f\circ h)$ and $f\circ 0 = 0$.
\end{definition}

\begin{definition}[{\citep[Definition~1.2.1]{MR2591951}}]
  A category is \emph{cartesian left-additive} if it is a
  left-additive category with products such that all projections and
  pairings of additive morphisms are additive.
\end{definition}

\begin{remark}
  If $\C$ is a cartesian left-additive category, then the pairing
  \[
  \<-,-\> : \C(Z, X)\times \C(Z, Y)\to \C(Z, X\times Y)
  \]
  is additive; in other words, it satisfies $\<f+g, h+k\> = \<f, h\> +
  \<g, k\>$ and $\<0, 0\>=0$.  For example, the equation $\<f+g, h+k\>
  = \<f, h\> + \<g, k\>$ follows from the equations
  \begin{alignat*}{3}
    \pi_1\circ (\<f, h\> + \<g, k\>) & = \pi_1\circ\<f, h\> +
    \pi_1\circ\<g, k\>
    & \qquad & \textup{by the additivity of $\pi_1$}\\
    & = f + g
    & \qquad & \textup{by \eqref{eq:proj}},\\
    \pi_2\circ (\<f, h\> + \<g, k\>) & = \pi_2\circ\<f, h\> +
    \pi_2\circ\<g, k\>
    & \qquad & \textup{by the additivity of $\pi_2$}\\
    & = h + k & \qquad & \textup{by \eqref{eq:proj}},
  \end{alignat*}
  and from the universal property of pairing.  The proof of the
  equation $\<0, 0\>=0$ is similar.  Also, note that in a cartesian
  left-additive category $0: X\to\unito$ is necessarily equal to
  $!_X$.
\end{remark}

\begin{definition}[{\citep[Section~1.4, Definition~2.1.1]{MR2591951}},
                   {\citep[Definition~4.2]{1857313}}]
  A cartesian closed category is \emph{cartesian closed left-additive}
  if it is a cartesian left-additive category such that each currying
  map $\curry: \C(Z\times X,Y)\to\C(Z,X\Rightarrow Y)$ is additive:
  $\curry(f + g) = \curry(f) + \curry(g)$ and $\curry(0) = 0$.  A
  \emph{cartesian (closed) differential category} is a cartesian
  (closed) left-additive category equipped with an operator $D : \C(X,
  Y) \to \C(X\times X, Y)$ satisfying the following axioms:
  \begin{itemize}
  \item[D1.] $D(f + g) = D(f) + D(g)$ and $D(0) = 0$.
  \item[D2.] $D(f) \circ \< h+k, v\> = D(f) \circ \< h, v\> + D(f)
    \circ \< k, v\>$ and $D(f) \circ \< 0, v\> = 0$.
  \item[D3.] $D(\id) = \pi_1$, $D(\pi_1) = \pi_1 \circ \pi_1$,
    $D(\pi_2) = \pi_2 \circ \pi_1$.
  \item[D4.] $D(\< f, g\>) = \< D(f), D(g)\>$.
  \item[D5.] $D(f\circ g) = D(f) \circ \< D(g), g \circ \pi_2 \>$.
  \item[D6.] $D(D(f)) \circ \<\< g, 0\>, \<h, k\>\> = D(f) \circ \< g,
    k\>$.
  \item[D7.] $D(D(f))\circ \<\<0, h\>, \<g, k\>\> = D(D(f))\circ
    \<\<0, g\>, \<h,k\>\>$.
  \end{itemize}
\end{definition}

Following \citet{MR2591951}, we suggest that the reader keep in mind
one key simple example of a cartesian differential category while
reading this note: the category of smooth maps, whose objects are
natural numbers and morphisms $m\to n$ are smooth maps $\R^m\to\R^n$.
The operator $D$ takes an $f: \R^m\to\R^n$ and produces a $D(f):
\R^m\times\R^m\to\R^n$ given by $D(f)(x', x) = J_f(x)\cdot x'$, where
$J_f(x)$ is the Jacobian of $f$ at the point $x$.  Be aware, however,
that this category is not closed, and hence is not a differential
$\lambda$-category.

Let us provide some intuition for the axioms: D1 says $D$ is linear;
D2 that $D(f)$ is additive in its first coordinate; D3 and D4 assert that
$D$ is compatible with the product structure, and D5 is the chain
rule.  We refer the reader to \citep[Lemma~2.2.2]{MR2591951} for the
proof that D6 is essentially requiring that $D(f)$ be linear (in the
sense defined below) in its first variable.  D7 is essentially
independence of order of partial differentiation.

Axiom D4 asserts that $D$ commutes with pairing.  We shall also need
the following formula for the derivative of a product.

\begin{lemma}\label{lem-D-times}
  $D(f\times g) = \<D(f)\circ\<\pi_1\circ\pi_1, \pi_1\circ\pi_2\>,
  D(g)\circ\<\pi_2\circ\pi_1, \pi_2\circ\pi_2\>\>$.
\end{lemma}

\begin{proof}
  By \eqref{eq:times-pair}, $f\times g = \<f\circ\pi_1,
  g\circ\pi_2\>$.  Therefore $D(f\times g) = D(\<f\circ\pi_1,
  g\circ\pi_2\>)= \<D(f\circ\pi_1), D(g\circ\pi_2)\>$ by D4.  Applying
  axioms D5 and D3 concludes the proof.
\end{proof}

We shall also need the following interchange property of the operator
$D$ that is slightly more general than D7.

\begin{lemma}\label{lemma-D-interchange}
  $D(D(f)) \circ \<\<i, h\>, \<g, k\>\> = D(D(f)) \circ \<\<i, g\>,
  \<h, k\>\>$.
\end{lemma}

\begin{proof}
  We have:
  \begin{alignat*}{3}
    D(D(f)) &\circ \<\<i,h\>, \<g,k\>\>
    \\
    & = D(D(f)) \circ \<\<i,0\> + \<0,h\>, \<g,k\>\>
    & \qquad & \textup{because pairing is additive}
    \displaybreak[0]
    \\
    & = D(D(f)) \circ \<\<i,0\>, \<g,k\>\>
      + D(D(f))\circ \<\<0, h\>, \<g, k\>\>
    & \qquad & \textup{by D2}
    \displaybreak[0]
    \\
    & = D(f) \circ \<i, k\> + D(D(f))\circ\<\<0,h\>, \<g,k\>\>
    & \qquad & \textup{by D6},
    \\
    \displaybreak[0]
    \\
    D(D(f)) &\circ \<\<i,g\>, \<h, k\>\>
    \\
    & = D(D(f)) \circ \<\<i,0\> + \<0, g\>, \<h,k\>\>
    & \qquad & \textup{because pairing is additive}
    \displaybreak[0]
    \\
    & = D(D(f)) \circ \<\<i,0\>, \<h,k\>\>
      + D(D(f)) \circ \<\<0,g\>, \<h, k\>\>
    & \qquad & \textup{by D2}
    \displaybreak[0]
    \\
    & = D(f) \circ \<i,k\> + D(D(f))\circ \<\<0, g\>, \<h,k\>\>
    & \qquad & \textup{by D6}.
  \end{alignat*}
  The right hand sides are equal by D7, hence the equality of the left
  hand sides.
\end{proof}

\begin{corollary}\label{cor-D-interchange}
  $D(D(f))\circ\sigma = D(D(f))$.
\end{corollary}

\begin{proof}
  Apply \lemref{lemma-D-interchange} to $i = \pi_1\circ\pi_1$, $g =
  \pi_1\circ\pi_2$, $h = \pi_2\circ\pi_1$, and $k = \pi_2\circ\pi_2$,
  and observe that $\<\<i, h\>, \<g, k\>\> = \<\<\pi_1\circ\pi_1,
  \pi_2\circ\pi_1\>, \<\<\pi_1\circ\pi_2, \pi_2\circ\pi_2\>\> =
  \<\<\pi_1, \pi_2\>\circ\pi_1, \<\pi_1, \pi_2\>\circ\pi_2\> =
  \<\pi_1, \pi_2\> = \id$ by \eqref{eq:circ-pair} and
  \eqref{eq:id-pair}.
\end{proof}

Following \citet{MR2591951}, we say that a morphism $f$ is
\emph{linear} if $D(f)=f\circ\pi_1$.  By\citep[Lemma~2.2.2]{MR2591951},
the class of linear morphisms is closed under sum, composition,
pairing, and product, and contains all identities, projections, and
zero morphisms.  This often allows us to argue that a morphism is
linear simply by inspection and to conclude that its derivative is
obtained by precomposing with the projection $\pi_1$.  For example,
the unit and associativity constraints $\ell$, $r$, and $a$ are
linear, as well as the symmetry $c$ and the distributivity isomorphism
$\sigma$.  Also, axiom D2 implies that any linear morphism is
additive.

\subsection{Differential $\lambda$-categories}

The notion of cartesian differential category was partly motivated by
the desire to model the differential $\lambda$-calculus of
\citet{MR2016523} categorically.  \citet{MR2591951} proved that
cartesian differential categories are sound and complete to model
suitable term calculi.  However, the properties of cartesian
differential categories are too weak for modeling the full
differential $\lambda$-calculus because the differential operator is
not necessarily compatible with the cartesian closed structure.  For
this reason, \citet{1857313} introduced the notion of differential
$\lambda$-category.

\begin{definition}[{\citep[Definition~4.4]{1857313}}]
  A \emph{differential $\lambda$-category} is a cartesian closed
  differential category such that for each $f : Z\times X\to Y$ holds
  \begin{equation}
    \label{eq:D-curry}
    D(\curry(f))=\curry(D(f)\circ\<\pi_1\times 0_X, \pi_2\times\id_X\>).
  \end{equation}
\end{definition}
We show in \proref{prop-diff-lambda-cat} that it suffices to check
this condition only for the evaluation morphisms.

\begin{example}[Convenient differential $\lambda$-category]
  \citet{BluteEhrhardTasson} proved that the category $C^\infty$ of
  convenient vector spaces and smooth maps is a cartesian closed
  differential category.  We are going to show that it is in fact a
  differential $\lambda$-category.

  We begin by recalling the notion of convenient vector space,
  following \citet{MR1471480}.  Let $E$ be a locally convex vector
  space.  A curve $c : \R\to E$ is called \emph{differentiable} if the
  \emph{derivative} $c'(t) = \lim_{s\to 0}\frac{1}{s}(c(t+s)-c(t))$ at
  $t$ exists for all $t$.  A curve $c : \R\to E$ is called
  \emph{smooth} if all iterated derivatives exist.  Let
  $\mathcal{C}_E$ denote the set of all smooth curves into $E$.  A
  locally convex vector space $E$ is called \emph{convenient} if it
  satisfies any of the equivalent conditions of
  \citep[Theorem~2.14]{MR1471480}.  In particular, $E$ is convenient
  if the following holds: for any curve $c : \R\to E$, if the
  composites $\ell\circ c : \R\to\R$ are smooth for all $\ell\in E^*$,
  then $c$ is smooth; here $E^*$ denotes the space of all continuous
  linear functionals on $E$.  A map $f : E\to F$ between convenient
  vector spaces is called \emph{smooth} if it maps smooth curves into
  $E$ to smooth curves into $F$; that is, if $f\circ
  c\in\mathcal{C}_F$ for all $c\in\mathcal{C}_E$.

  Let $C^\infty$ denote the category of convenient vector spaces and
  smooth maps.  \citeauthor{MR1471480} proved
  \citep[Theorem~3.12]{MR1471480} that the category $C^\infty$ is
  cartesian closed.  For a pair of convenient vector spaces $E$ and
  $F$, the exponential object $E\Rightarrow F$ is the locally convex
  space $C^\infty(E, F)$ of all smooth mappings $E\to F$ with
  pointwise linear structure and the initial topology with respect to
  all mappings $c^* : C^\infty(E, F)\to C^\infty(\R, F)$, $f\mapsto
  f\circ c$, for $c\in\mathcal{C}_E$, where each space $C^\infty(\R,
  F)$ is given the topology of uniform convergence on compact sets of
  each derivative separately.  \citet{BluteEhrhardTasson} proved that
  $C^\infty$ is a cartesian differential category.  The differential
  operator $D : C^\infty(E, F)\to C^\infty(E\times E, F)$ is given by
  $D(f)(x', x) = c'(0)$, where $c : \R\to F$ is the smooth curve into
  $F$ given by $c(t) = f(x + tx')$.

  Let us show that $C^\infty$ is a differential $\lambda$-category.
  By \proref{prop-diff-lambda-cat}, it suffices to show that
  equation~\eqref{eq:D-uncurry-2} holds, i.e.,
  \[
  \ev\circ (\pi_1\times\id_E) = D(\ev)\circ \<\pi_1\times 0,
  \pi_2\times \id_E\> : (C^\infty(E,F)\times C^\infty(E,F))\times
  E\to F.
  \]
  Let $((f, g), x)\in (C^\infty(E,F)\times C^\infty(E,F))\times E$ be
  an arbitrary point.  Evaluating the right hand side of the equation
  at the point $((f, g), x)$, we obtain:
  \begin{align*}
  \lim_{t\to 0}\frac{\ev((g, x) + t(f, 0)) - \ev(g, x)}{t}
  & = \lim_{t\to 0}\frac{\ev(g+tf, x)-\ev(g, x)}{t}
  \\
  & = \lim_{t\to 0}\frac{(g+tf)(x) - g(x)}{t}
  \\
  & = \lim_{t\to 0}\frac{g(x) + tf(x) - g(x)}{t}
  \\
  & = \lim_{t\to 0}f(x)
  \\
  & = f(x),
  \end{align*}
  which obviously coincides with the value of the left hand side at
  the point $((f, g), x)$, hence the assertion.
\end{example}

The reader is referred to \citep{1857313} for two other examples of
differential $\lambda$-categories.

\section{Tangent bundle}

The differential operator $D$ allows us to replicate the construction
of the tangent bundle of a smooth manifold from differential geometry
in any cartesian differential category.  In this section, we define
the tangent bundle functor $T$ on a cartesian differential category
$\C$ and study its properties.  We prove in
Sections~\ref{sec-monad-structure} and \ref{sec-tensorial-strength}
that $T$ can be equipped with the structure of a strong monad.
Furthermore, we show in Section~\ref{sec-monoidal-structure} that
this monad is commutative.  We prove in
Section~\ref{sec-distributive-law} that the natural transformation
known in differential geometry as ``canonical flip'' is a distributive
law of the monad $T$ over itself.  Starting from
Section~\ref{sec-monad-on-diff-cat} we assume that $\C$ is a cartesian
closed differential category, and in fact a differential
$\lambda$-category.  We study the closed structure and the enrichment
of the functor $T$ in Sections~\ref{sec-closed-structure} and
\ref{sec-enrichment} respectively.

\subsection{The tangent bundle functor $T$}

Let $\C$ be a cartesian differential category.  The \emph{tangent
  bundle functor} $T : \C \to \C$ is defined by $TX = X\times X$ and
$T(f) = \<D(f), f\circ\pi_2\>$.  Intuitively, $TX$ is a set of pairs
$(x', x)$ consisting of a point $x\in X$ and a tangent vector $x'\in
X$ at the point $x$.  Set-theoretically, $T(f)$ is given by $T(f)(x',
x) = (D(f)(x', x), f(x))$.  Here $D(f)(x', x)$ plays the role of the
Jacobian of $f$ at the point $x$ multiplied by the vector $x'$.

\begin{lemma}
  $T$ is a functor.
\end{lemma}

\begin{proof}
  Let us check that $T$ preserves identities and composition.  We
  have:
  \begin{alignat*}{3}
    T(\id)
    &= \< D(\id), \id\circ\pi_2\>
    & \qquad & \textup{by def. of $T$}
    \displaybreak[0]
    \\
    &= \<\pi_1, \pi_2\>
    & \qquad & \textup{by D3}
    \displaybreak[0]
    \\
    &= \id
    & \qquad & \textup{by \eqref{eq:id-pair}},
    \\
    \displaybreak[0]
    \\
    T(f\circ g)
    &= \< D(f\circ g), (f\circ g)\circ\pi_2\>
    & \qquad & \textup{by def. of $T$}
    \displaybreak[0]
    \\
    &= \< D(f)\circ \<D(g), g\circ \pi_2\>, f\circ (g\circ\pi_2)\>
    & \qquad & \textup{by D5}
    \displaybreak[0]
    \\
    &= \< D(f)\circ \<D(g), g\circ \pi_2\>,
    f\circ \pi_2 \circ \<D(g), g\circ\pi_2\>\>
    & \qquad & \textup{by \eqref{eq:proj}}
    \displaybreak[0]
    \\
    &=\<D(f), f\circ\pi_2\> \circ \<D(g), g\circ\pi_2\>
    & \qquad & \textup{by \eqref{eq:circ-pair}}
    \displaybreak[0]
    \\
    &= T(f) \circ T(g)
    & \qquad & \textup{by def. of $T$}.
  \end{alignat*}
  The lemma is proven.
\end{proof}

\begin{lemma}\label{lemma-T-additive-functor}
  $T$ is an additive functor.
\end{lemma}

\begin{proof}
  Let us check that for any objects $X$ and $Y$ the map $T : \C(X, Y)
  \to \C(TX, TY)$ is additive.  We have:
  \begin{alignat*}{3}
    T(0)
    &= \<D(0), 0 \circ \pi_2\>
    & \qquad & \textup{by def. of $T$}
    \displaybreak[0]
    \\
    &= \<0, 0\>
    & \qquad & \textup{by D1 and because $\C$ is left-additive}
    \displaybreak[0]
    \\
    &= 0
    & \qquad & \textup{because pairing is additive},
    \\
    \displaybreak[0]
    \\
    T(f+g)
    &= \<D(f+g), (f+g)\circ\pi_2\>
    & \qquad & \textup{by def. of $T$}
    \displaybreak[0]
    \\
    &= \<D(f) + D(g), (f\circ\pi_2) + (g\circ\pi_2)\>
    & \qquad & \textup{by D1 and because $\C$ is left-additive}
    \displaybreak[0]
    \\
    &= \<D(f), f\circ\pi_2\> + \<D(g), g\circ\pi_2\>
    & \qquad & \textup{because pairing is additive}
    \displaybreak[0]
    \\
    &= T(f) + T(g)
    & \qquad & \textup{by def. of $T$}.
  \end{alignat*}
  The lemma is proven.
\end{proof}

\begin{lemma}\label{lemma-T-additive-morphisms}
  If $f$ is linear, then $T(f) = f\times f$.
\end{lemma}

\begin{proof}
  Follows from the definition of $T$ and \eqref{eq:times-pair}.
\end{proof}

\subsection{The monad structure on $T$}
\label{sec-monad-structure}

Let us show that the tangent bundle functor $T$ is part of a monad.
The unit and multiplication are defined as follows.  For each object
$X$ of $\C$, we denote by $\eta_X$ the morphism $\<0, \id\> : X\to
X\times X= TX$ and by $\mu_X$ the morphism
\[
\<\pi_2\circ\pi_1 + \pi_1\circ\pi_2, \pi_2\circ\pi_2\>:
TTX = (X \times X) \times (X\times X) \to X \times X = TX.
\]
Set-theoretically, $\eta_X(x) = (0, x)$ and $\mu_X((w, v),(u, x)) = (v
+ u, x)$.  Clearly, $\eta_X$ and $\mu_X$ are linear morphisms.

\begin{lemma}
  $\eta$ is a natural transformation $\Id\to T$.
\end{lemma}

\begin{proof}
  For any morphism $f : X\to Y$, we have:
  \begin{alignat*}{3}
    T(f)\circ \eta_X
    &= \<D(f), f\circ\pi_2\>\circ\<0, \id\>
    & \qquad & \textup{by def. of $T$ and $\eta$}
    \displaybreak[0]
    \\
    &= \<D(f)\circ\<0, \id\>, f\circ\pi_2\circ\<0, \id\>\>
    & \qquad & \textup{by \eqref{eq:circ-pair}}
    \displaybreak[0]
    \\
    &= \<0, f\>
    & \qquad & \textup{by D2 and \eqref{eq:proj}}
    \displaybreak[0]
    \\
    &= \<0\circ f, \id\circ f\>
    & \qquad & \textup{because $\C$ is left-additive}
    \displaybreak[0]
    \\
    &= \<0, \id\> \circ f
    & \qquad & \textup{by \eqref{eq:circ-pair}}
    \displaybreak[0]
    \\
    &= \eta_Y \circ f
    & \qquad & \textup{by def. of $\eta$},
  \end{alignat*}
  hence the assertion.
\end{proof}

\begin{lemma}
  $\mu$ is a natural transformation $TT\to T$.
\end{lemma}

\begin{proof}
  Let $f : X\to Y$ be a morphism in $\C$.  We have, on the one hand:
  \begin{align*}
    & T(f)\circ\mu_X
    \\[2mm]
    ={} & \quad\textup{by def. of $T$ and $\mu$}
    \\[2mm]
    & \<D(f), f \circ \pi_2\> \circ \<\pi_2 \circ \pi_1 + \pi_1 \circ
    \pi_2, \pi_2 \circ \pi_2\>
    \displaybreak[0]
    \\[2mm]
    ={} & \quad\textup{by \eqref{eq:circ-pair}}
    \\[2mm]
    & \<D(f) \circ \<\pi_2 \circ \pi_1 + \pi_1 \circ \pi_2, \pi_2
    \circ \pi_2 \>, f \circ \pi_2 \circ \<\pi_2 \circ \pi_1 + \pi_1
    \circ \pi_2, \pi_2 \circ \pi_2\>\>
    \displaybreak[0]
    \\[2mm]
    ={} & \quad\textup{by D2 and \eqref{eq:proj}}
    \\[2mm]
    & \<D(f) \circ \<\pi_2 \circ \pi_1, \pi_2 \circ \pi_2\> + D(f)
    \circ \<\pi_1 \circ \pi_2, \pi_2 \circ \pi_2\>, f \circ \pi_2
    \circ \pi_2\>
    \displaybreak[0]
    \\[2mm]
    ={} & \quad\textup{by \eqref{eq:circ-pair} and \eqref{eq:id-pair}}
    \\[2mm]
    & \<D(f) \circ \<\pi_2 \circ \pi_1, \pi_2 \circ \pi_2\> + D(f)
    \circ \pi_2, f\circ \pi_2 \circ \pi_2\>.
\end{align*}
On the other hand:
\begin{align*}
    & \mu_Y \circ T(T(f))
    \\[2mm]
    ={} & \quad\textup{by def. of $T$ and $\mu$}
    \\[2mm]
    & \<\pi_2 \circ \pi_1 + \pi_1 \circ \pi_2,
    \pi_2 \circ \pi_2\> \circ \<D(T(f)), T(f) \circ \pi_2\>
    \displaybreak[0]
    \\[2mm]
    ={} & \quad\textup{by \eqref{eq:circ-pair}}
    \\[2mm]
    & \<(\pi_2 \circ \pi_1 + \pi_1 \circ \pi_2) \circ \<D(T(f)),
    T(f) \circ \pi_2\>, \pi_2 \circ \pi_2 \circ \<D(T(f)), T(f) \circ
    \pi_2\>\>
    \displaybreak[0]
    \\[2mm]
    ={} & \quad\textup{by \eqref{eq:proj} and because $\C$ is left-additive}
    \\[2mm]
    & \< \pi_2\circ\pi_1\circ \< D(T(f)), T(f)\circ \pi_2\> +
    \pi_1\circ\pi_2\circ \< D(T(f)), T(f)\circ \pi_2\>, \pi_2\circ
    T(f) \circ \pi_2\>
    \displaybreak[0]
    \\[2mm]
    ={} & \quad\textup{by \eqref{eq:proj}}
    \\[2mm]
    & \<\pi_2 \circ D(T(f)) + \pi_1 \circ T(f) \circ \pi_2, \pi_2
    \circ T(f) \circ \pi_2\>
    \displaybreak[0]
    \\[2mm]
    ={} & \quad\textup{by def. of $T$, D4, and \eqref{eq:proj}}
    \\[2mm]
    & \<D(f\circ\pi_2) + D(f) \circ \pi_2, f\circ \pi_2 \circ \pi_2\>
    \displaybreak[0]
    \\[2mm]
    ={} & \quad\textup{by D5 and D3}
    \\[2mm]
    & \<D(f) \circ \<\pi_2 \circ \pi_1, \pi_2 \circ \pi_2\> + D(f)
    \circ \pi_2, f\circ \pi_2 \circ \pi_2\>.
  \end{align*}
  The obtained expressions are identical, hence the naturality of
  $\mu$.
\end{proof}

Let us prove that the triple $(T, \eta, \mu)$ is a monad on the
category $\C$.  The proof consists of checking the monad axioms.

\begin{lemma}\label{lem-left-unit}
  The natural transformations $\eta$ and $\mu$ satisfy the equation
  $\mu \circ \eta T = \id$.
\end{lemma}

\begin{proof}
  We need to show that for each object $X$ of $\C$ holds
  $\mu_X\circ\eta_{TX}=\id_{TX}$.  We have:
  \begin{alignat*}{3}
    \mu_X \circ \eta_{TX}
    &= \<\pi_2\circ\pi_1 + \pi_1\circ\pi_2,
         \pi_2\circ\pi_2\>\circ\<0, \id\>
    & \qquad & \textup{by def. of $\eta$ and $\mu$}
    \displaybreak[0]
    \\
    &=\<(\pi_2\circ\pi_1 + \pi_1\circ\pi_2)\circ\<0, \id\>,
        \pi_2\circ\pi_2\circ\<0, \id\>\>
    & \qquad & \textup{by \eqref{eq:circ-pair}}
    \displaybreak[0]
    \\
    &=\<\pi_2\circ\pi_1\circ\<0,\id\> +
        \pi_1\circ\pi_2\circ\<0,\id\>,
        \pi_2\circ\pi_2\circ\<0,\id\>\>
    & \qquad & \textup{because $\C$ is left-additive}
    \displaybreak[0]
    \\
    &=\<\pi_2\circ 0 + \pi_1\circ\id, \pi_2\circ\id\>
    & \qquad & \textup{by \eqref{eq:proj}}
    \displaybreak[0]
    \\
    &=\<\pi_1, \pi_2\>
    & \qquad & \textup{because $\pi_2$ is additive}
    \displaybreak[0]
    \\
    &=\id
    & \qquad & \textup{by \eqref{eq:id-pair}.}
  \end{alignat*}
  The lemma is proven.
\end{proof}

\begin{lemma}\label{lem-right-unit}
  The natural transformations $\eta$ and $\mu$ satisfy the equation
  $\mu \circ T\eta = \id$.
\end{lemma}

\begin{proof}
  We need to show that for each object $X$ of $\C$ holds $\mu_X\circ
  T(\eta_X) = \id_{TX}$.  Since $\eta_X$ is linear, it follows by
  \lemref{lemma-T-additive-morphisms} that
  $T(\eta_X)=\eta_X\times\eta_X$.  We have:
  \begin{alignat*}{3}
    \mu_X\circ T(\eta_X)
    &= \<\pi_2\circ\pi_1 + \pi_1\circ\pi_2,
         \pi_2\circ\pi_2\> \circ (\eta\times\eta)
    & \qquad & \textup{by def. of $\mu$}
    \displaybreak[0]
    \\
    &= \<(\pi_2\circ\pi_1 + \pi_1\circ\pi_2) \circ (\eta\times\eta),
         \pi_2\circ\pi_2\circ(\eta\times\eta)\>
    & \qquad & \textup{by \eqref{eq:circ-pair}}
    \displaybreak[0]
    \\
    &= \<\pi_2\circ\pi_1\circ(\eta\times\eta) +
         \pi_1\circ\pi_2\circ(\eta\times\eta),
         \pi_2\circ\pi_2\circ(\eta\times\eta)\>
    & \qquad & \textup{because $\C$ is left-additive}
    \displaybreak[0]
    \\
    &= \<\pi_2\circ\eta\circ\pi_1 + \pi_1\circ\eta\circ\pi_2,
         \pi_2\circ\eta\circ\pi_2\>
    & \qquad & \textup{by \eqref{eq:times-def}}
    \displaybreak[0]
    \\
    &= \<\id\circ\pi_1 + 0\circ\pi_2, \id\circ\pi_2\>
    & \qquad & \textup{by def. of $\eta$ and \eqref{eq:proj}}
    \displaybreak[0]
    \\
    &= \<\pi_1, \pi_2\>
    & \qquad & \textup{because $\C$ is left-additive}
    \displaybreak[0]
    \\
    &= \id
    & \qquad & \textup{by \eqref{eq:id-pair}}.
  \end{alignat*}
  The lemma is proven.
\end{proof}

\begin{lemma}\label{lem-assoc}
  The natural transformation $\mu$ satisfies the equation $\mu\circ
  \mu T = \mu \circ T\mu$.
\end{lemma}

\begin{proof}
  We need to show that for each object $X$ of $\C$ holds
  $\mu_X\circ\mu_{TX} = \mu_X\circ T(\mu_X)$.  We have:
  \begin{align*}
    & \mu_X\circ\mu_{TX}
    \\[2mm]
    ={} & \quad\textup{by def. of $\mu$}
    \\[2mm]
    & \<\pi_2\circ\pi_1 + \pi_1\circ\pi_2, \pi_2\circ\pi_2\> \circ
      \<\pi_2\circ\pi_1 + \pi_1\circ\pi_2, \pi_2\circ\pi_2\>
    \displaybreak[0]
    \\[2mm]
    ={} & \quad\textup{by \eqref{eq:circ-pair}}
    \\[2mm]
    & \<(\pi_2\circ\pi_1 + \pi_1\circ\pi_2) \circ \<\pi_2\circ\pi_1 +
    \pi_1\circ\pi_2, \pi_2\circ\pi_2\>, \pi_2\circ\pi_2 \circ
    \<\pi_2\circ\pi_1 + \pi_1\circ\pi_2, \pi_2\circ\pi_2\>\>
    \displaybreak[0]
    \\[2mm]
    ={} & \quad\textup{because $\C$ is left-additive}
    \\[2mm]
    & \<\pi_2\circ\pi_1\circ \<\pi_2\circ\pi_1 + \pi_1\circ\pi_2,
    \pi_2\circ\pi_2\> + \pi_1\circ\pi_2\circ \<\pi_2\circ\pi_1 +
    \pi_1\circ\pi_2, \pi_2\circ\pi_2\>,
    \\
    & \hskip0.4em\pi_2\circ\pi_2 \circ \<\pi_2\circ\pi_1 +
    \pi_1\circ\pi_2, \pi_2\circ\pi_2\>\>
    \displaybreak[0]
    \\[2mm]
    ={} & \quad\textup{by \eqref{eq:proj}}
    \\[2mm]
    & \<\pi_2\circ(\pi_2\circ\pi_1 + \pi_1\circ\pi_2) +
    \pi_1\circ\pi_2\circ\pi_2, \pi_2\circ\pi_2\circ\pi_2\>
    \displaybreak[0]
    \\[2mm]
    ={} & \quad\textup{because $\pi_2$ is additive}
    \\[2mm]
    & \<\pi_2\circ\pi_2\circ\pi_1 + \pi_2\circ\pi_1\circ\pi_2 +
    \pi_1\circ\pi_2\circ\pi_2, \pi_2\circ\pi_2\circ\pi_2\>.
    \intertext{%
      On the other hand, because $\mu_X$ is linear and
      consequently $T(\mu_X)=\mu_X\times\mu_X$ by
      \lemref{lemma-T-additive-morphisms}, we have:
    }%
    & \mu_X\circ
    T(\mu_X)
    \\[2mm]
    ={} & \quad\textup{by def. of $\mu$}
    \\[2mm]
    & \<\pi_2\circ\pi_1 + \pi_1\circ\pi_2, \pi_2\circ\pi_2\> \circ
    T(\mu_X)
    \displaybreak[0]
    \\[2mm]
    ={} & \quad\textup{by \eqref{eq:circ-pair} and because $\C$ is
      left-additive}
    \\[2mm]
    & \<\pi_2\circ\pi_1\circ T(\mu_X) + \pi_1\circ\pi_2\circ T(\mu_X),
    \pi_2\circ\pi_2\circ T(\mu_X)\>
    \displaybreak[0]
    \\[2mm]
    ={} & \quad\textup{by \eqref{eq:times-def} and because
      $T(\mu_X)=\mu_X\times\mu_X$}
    \\[2mm]
    & \<\pi_2\circ\mu_X\circ\pi_1 + \pi_1\circ\mu_X\circ\pi_2,
    \pi_2\circ\mu_X\circ\pi_2\>
    \displaybreak[0]
    \\[2mm]
    ={} & \quad\textup{by def. of $\mu_X$ and \eqref{eq:proj}}
    \\[2mm]
    & \<\pi_2\circ\pi_2\circ\pi_1 + (\pi_2\circ\pi_1 +
    \pi_1\circ\pi_2)\circ\pi_2, \pi_2\circ\pi_2\circ\pi_2\>
    \displaybreak[0]
    \\[2mm]
    ={} & \quad\textup{because $\C$ is left-additive}
    \\[2mm]
    & \<\pi_2\circ\pi_2\circ\pi_1 + \pi_2\circ\pi_1\circ\pi_2 +
    \pi_1\circ\pi_2\circ\pi_2, \pi_2\circ\pi_2\circ\pi_2\>,
  \end{align*}
  which coincides with the expression we obtained above for
  $\mu_X\circ\mu_{TX}$.  The lemma is proven.
\end{proof}

\begin{theorem}
  The triple $(T, \eta, \mu)$ is a monad on the category $\C$.
\end{theorem}

\begin{proof}
  Follows from Lemmas~\ref{lem-left-unit}, \ref{lem-right-unit},
  \ref{lem-assoc}.
\end{proof}

\subsection{The tensorial strength of $T$}
\label{sec-tensorial-strength}

We recall (e.g., from \citep[Definition~3.2]{MR1115262}) that a monad
$(T, \eta, \mu)$ is \emph{strong} if it is equipped with a
\emph{tensorial strength}, a natural transformation $t_{X, Y} :
X\times TY \to T(X\times Y)$, such that the diagrams
\begin{minipage}[b]{0.5\textwidth}
\begin{equation}
  \label{eq:strength-terminal}
  \begin{xy}
    (-20, -9)*+{\unito\times TX}="1"; (20, -9)*+{T(\unito\times
      X)}="2"; (0, 9)*+{TX}="3"; {\ar@{->}^-{t_{\unito,X}} "1";"2"};
    {\ar@{->}_-{\ell_{TX}} "3";"1"}; {\ar@{->}^-{T(\ell_X)}
      "3";"2"};
  \end{xy}
\end{equation}
\begin{equation}
  \label{eq:strength-assoc}
  \hspace{-0.7em}
  \begin{xy}
    (-20, 18)*+{(X\times Y)\times TZ}="1"; (20, 18)*+{T((X\times
      Y)\times Z)}="2"; (-20, 0)*+{X\times (Y\times TZ)}="3"; (-20,
    -18)*+{X\times T(Y\times Z)}="4"; (20, -18)*+{T(X\times (Y\times
      Z))}="5"; {\ar@{->}^-{t_{X\times Y, Z}} "1";"2"};
    {\ar@{->}_-{a_{X,Y,TZ}} "1";"3"}; {\ar@{->}_-{\id_X\times
        t_{Y,Z}} "3";"4"}; {\ar@{->}^-{T(a_{X,Y,Z})} "2";"5"};
    {\ar@{->}^-{t_{X,Y\times Z}} "4";"5"};
  \end{xy}
\end{equation}
\end{minipage}
\begin{minipage}[b]{0.5\textwidth}
\begin{equation}
  \label{eq:strength-unit}
  \begin{xy}
    (-20, 9)*+{X\times Y}="1"; (20, 9)*+{X\times TY}="2"; (0,
    -9)*+{T(X\times Y)}="3"; {\ar@{->}^-{\id_X\times\eta_Y}
      "1";"2"}; {\ar@{->}_-{\eta_{X\times Y}} "1";"3"};
    {\ar@{->}^-{t_{X,Y}} "2";"3"};
  \end{xy}
\end{equation}
\begin{equation}
  \label{eq:strength-mult}
  \begin{xy}
    (-20, 18)*+{X\times TTY}="1"; (20, 18)*+{T(X\times TY)}="2";
    (-20, -18)*+{X\times TY}="3"; (20, 0)*+{TT(X\times Y)}="4"; (20,
    -18)*+{T(X\times Y)}="5"; {\ar@{->}^-{t_{X,TY}} "1";"2"};
    {\ar@{->}_-{\id_X\times\mu_Y} "1";"3"}; {\ar@{->}^-{T(t_{X,Y})}
      "2";"4"}; {\ar@{->}^-{\mu_{X\times Y}} "4";"5"};
    {\ar@{->}^-{t_{X,Y}} "3";"5"};
  \end{xy}
\end{equation}
\end{minipage}
commute.  The prominence of strong monads in functional programming
has become apparent after the seminal work of \citet{MR1115262} on
computational $\lambda$-calculi, in which he suggested strong monads
as an appropriate way to model computations.

We are going to show that the tangent bundle monad $(T, \eta, \mu)$ is
strong.  The tensorial strength $t$ is defined as follows.  For any
pair of objects $X$ and $Y$ of $\C$, denote by $t_{X,Y}$ the morphism
\begin{equation}
  \label{eq:strength-def}
  \<\<0, \pi_1\circ\pi_2\>, \<\pi_1, \pi_2\circ\pi_2\>\> : X\times TY =
  X\times (Y\times Y) \to (X\times Y) \times (X\times Y) = T(X\times Y).
\end{equation}
By \eqref{eq:times-def} and the left-additivity of $\C$, we can also
write $t_{X,Y}$ as $\<0\times\pi_1,\id_X\times\pi_2\>$.
Set-theoretically, $t_{X,Y}(x, (y', y)) = ((0, y'), (x, y))$.
Intuitively, $t_{X,Y}$ assigns to a point $x\in X$ and a tangent
vector $y'\in Y$ at a point $y\in Y$ the vector $y'$ viewed as the
tangent vector to $X\times Y$ at the point $(x,y)$ whose component
along the ``$X$ axis'' is zero.

\begin{lemma}
  $t$ is a natural transformation.
\end{lemma}

\begin{proof}
  We need to show that for any morphisms $f: X\to U$ and $g: Y\to V$,
  the following diagram commutes:
  \[
  \begin{xy}
    (-20, 9)*+{X\times TY}="1";
    (20, 9)*+{T(X\times Y)}="2";
    (-20, -9)*+{U\times TV}="3";
    (20, -9)*+{T(U\times V)}="4";
    {\ar@{->}^-{t_{X,Y}} "1";"2"};
    {\ar@{->}_-{f\times T(g)} "1";"3"};
    {\ar@{->}^-{T(f\times g)} "2";"4"};
    {\ar@{->}^-{t_{U,V}} "3";"4"};
  \end{xy}
  \]
  We have, on the one hand:
  \begin{align*}
    & t_{U,V}\circ (f\times T(g))
    \\[2mm]
    ={} & \quad\textup{by def. of $t$}
    \\[2mm]
    & \<\<0, \pi_1\circ\pi_2\>, \<\pi_1,
    \pi_2\circ\pi_2\>\>\circ(f\times T(g))
    \displaybreak[0]
    \\[2mm]
    ={} & \quad\textup{by \eqref{eq:circ-pair}}
    \\[2mm]
    & \<\<0\circ (f\times T(g)), \pi_1\circ\pi_2\circ (f\times
    T(g))\>, \<\pi_1\circ(f\times T(g)), \pi_2\circ\pi_2\circ(f\times
    T(g))\>\>
    \displaybreak[0]
    \\[2mm]
    ={} & \quad\textup{by \eqref{eq:times-pair} and because $\C$ is
      left-additive}
    \\[2mm]
    & \<\<0, \pi_1\circ T(g)\circ \pi_2\>, \<f\circ\pi_1, \pi_2\circ
    T(g) \circ\pi_2\>\>
    \displaybreak[0]
    \\[2mm]
    ={} & \quad\textup{by def. of $T$}
    \\[2mm]
    & \<\<0, D(g)\circ \pi_2\>, \<f\circ\pi_1, g\circ\pi_2\circ\pi_2\>\>.
  \end{align*}
  On the other hand:
  \begin{align*}
    & T(f\times g) \circ t_{X, Y}
    \\[2mm]
    ={} & \quad\textup{by def. of $T$ and $t$}
    \\[2mm]
    & \<D(f\times g), (f\times g)\circ\pi_2\> \circ \<\<0,
    \pi_1\circ\pi_2\>, \<\pi_1, \pi_2\circ\pi_2\>\>
    \displaybreak[0]
    \\[2mm]
    ={} & \quad\textup{by \eqref{eq:circ-pair}}
    \\[2mm]
    & \<D(f\times g) \circ \<\<0, \pi_1\circ\pi_2\>, \<\pi_1,
    \pi_2\circ\pi_2\>\>, (f\times g)\circ \pi_2\circ \<\<0,
    \pi_1\circ\pi_2\>, \<\pi_1, \pi_2\circ\pi_2\>\>\>
    \displaybreak[0]
    \\[2mm]
    ={} & \quad\textup{by \eqref{eq:proj}}
    \\[2mm]
    & \<D(f\times g) \circ \<\<0, \pi_1\circ\pi_2\>, \<\pi_1,
    \pi_2\circ\pi_2\>\>, (f\times g)\circ \<\pi_1, \pi_2\circ\pi_2\>\>
    \displaybreak[0]
    \\[2mm]
    ={} & \quad\textup{by \eqref{eq:times-circ-pair}}
    \\[2mm]
    & \<D(f\times g) \circ \<\<0, \pi_1\circ\pi_2\>, \<\pi_1,
    \pi_2\circ\pi_2\>\>, \<f\circ \pi_1, g\circ\pi_2\circ\pi_2\>\>.
  \end{align*}
  Let us consider the first component of the last expression
  separately.  We need to show that it is equal to $\<0,
  D(g)\circ\pi_2\>$.  Indeed:
  \begin{align*}
    & D(f\times g)\circ \<\<0, \pi_1\circ\pi_2\>, \<\pi_1,
    \pi_2\circ\pi_2\>\>
    \\[2mm]
    ={} & \quad\textup{by \lemref{lem-D-times}}
    \\[2mm]
    & \<D(f) \circ \<\pi_1\circ\pi_1, \pi_1\circ\pi_2\> \circ \<\<0,
    \pi_1\circ\pi_2\>, \<\pi_1, \pi_2\circ\pi_2\>\>,
    \\
    & \hskip0.47em D(g) \circ \<\pi_2\circ\pi_1,
    \pi_2\circ\pi_2\>\circ \<\<0, \pi_1\circ\pi_2\>, \<\pi_1,
    \pi_2\circ\pi_2\>\>\>
    \displaybreak[0]
    \\[2mm]
    ={} & \quad\textup{by \eqref{eq:proj} and \eqref{eq:circ-pair}}
    \\[2mm]
    & \<D(f)\circ \<0, \pi_1\>, D(g) \circ\<\pi_1\circ\pi_2,
    \pi_2\circ\pi_2\>\>
    \displaybreak[0]
    \\[2mm]
    ={} & \quad\textup{by D2, \eqref{eq:circ-pair}, and
      \eqref{eq:id-pair}}
    \\[2mm]
    & \<0, D(g)\circ\pi_2\>,
  \end{align*}
  which completes the proof of the lemma.
\end{proof}

Let us check that $t$ is indeed a tensorial strength for the monad
$(T, \eta, \mu)$.  The proof consists of checking the commutativity of
diagrams~\eqref{eq:strength-terminal}--\eqref{eq:strength-mult}.

\begin{lemma}\label{lem-strength-terminal}
  $t$ makes diagram~\eqref{eq:strength-terminal} commute.
\end{lemma}

\begin{proof}
  We have:
  \begin{alignat*}{3}
    t_{\unito, X}\circ \ell_{TX}
    &= \<\<0, \pi_1\circ\pi_2\>, \<\pi_1, \pi_2\circ\pi_2\>\>
       \circ \<!_{TX},\id_{TX}\>
    & \qquad & \textup{by def. of $t_{\unito, X}$ and $\ell_{TX}$}
    \displaybreak[0]
    \\
    &=\<\<0\circ\<!_{TX}, \id_{TX}\>,
        \pi_1\circ\pi_2\circ\<!_{TX}, \id_{TX}\>\>,
    \\
    &\hphantom{{}=\<}
      \<\pi_1\circ\<!_{TX}, \id_{TX}\>,
        \pi_2\circ\pi_2\circ\<!_{TX}, \id_{TX}\>\>
    & \qquad & \textup{by \eqref{eq:circ-pair}}
    \displaybreak[0]
    \\
    &=\<\<0, \pi_1\>, \<!_{TX}, \pi_2\>\>
    & \qquad & \textup{by \eqref{eq:proj} and because $\C$ is
      left-additive}
    \\
    \displaybreak[0]
    \\
    T(\ell_X)
    &= \ell_X\times\ell_X
    & \qquad & \textup{by \lemref{lemma-T-additive-morphisms} because
      $\ell_X$ is linear}
    \displaybreak[0]
    \\
    &= \<\<!_X, \id_X\> \circ\pi_1, \<!_X, \id_X\>\circ\pi_2\>
    & \qquad & \textup{by def. of $\ell_X$ and by
      \eqref{eq:times-pair}}
    \displaybreak[0]
    \\
    &= \<\<!_{X\times X}, \pi_1\>, \<!_{X\times X}, \pi_2\>\>
    & \qquad & \textup{by \eqref{eq:circ-pair}}.
  \end{alignat*}
  The equation follows because in a cartesian left-additive category
  $0 : X\times X\to \unito$ is necessarily equal to $!_{X\times X}$.
\end{proof}

\begin{lemma}\label{lem-strength-assoc}
  $t$ makes diagram~\eqref{eq:strength-assoc} commute.
\end{lemma}

\begin{proof}
  We have, on the one hand:
  \begin{align*}
    & T(a_{X, Y, Z})\circ t_{X\times Y, Z}
    \\[2mm]
    ={} & \quad\textup{by \lemref{lemma-T-additive-morphisms}
      because $a_{X,Y,Z}$ is linear}
    \\[2mm]
    &(a_{X, Y, Z}\times a_{X, Y, Z})\circ t_{X\times Y, Z}
    \displaybreak[0]
    \\[2mm]
    ={} & \quad\textup{by def. of $t_{X\times Y, Z}$ and by
      \eqref{eq:times-circ-pair}}
    \\[2mm]
    & \<a_{X,Y,Z}\circ \<0, \pi_1\circ\pi_2\>, a_{X,Y,Z}\circ\<\pi_1,
    \pi_2\circ\pi_2\>\>
    \displaybreak[0]
    \\[2mm]
    ={} & \quad\textup{by def. of $a_{X,Y,Z}$ and by
      \eqref{eq:circ-pair}}
    \\[2mm]
    & \<\<\pi_1\circ\pi_1\circ\<0, \pi_1\circ\pi_2\>,
    \<\pi_2\circ\pi_1\circ\<0, \pi_1\circ\pi_2\>, \pi_2\circ\<0,
    \pi_1\circ\pi_2\>\>\>,
    \\
    & \hphantom{\<}\<\pi_1\circ\pi_1\circ\<\pi_1, \pi_2\circ\pi_2\>,
    \<\pi_2\circ\pi_1\circ\<\pi_1, \pi_2\circ\pi_2\>,
    \pi_2\circ\<\pi_1, \pi_2\circ\pi_2\>\>\>\>
    \displaybreak[0]
    \\[2mm]
    ={} & \quad\textup{by \eqref{eq:proj} and because $\pi_1$ and
      $\pi_2$ are additive}
    \\[2mm]
    & \<\<0, \<0, \pi_1\circ\pi_2\>\>, \<\pi_1\circ\pi_1,
    \<\pi_2\circ\pi_1, \pi_2\circ\pi_2\>\>\>.  \intertext{On the other
      hand:} & t_{X, Y\times Z}\circ (\id_X\times t_{Y,Z})\circ
    a_{X,Y,TZ}
    \displaybreak[0]
    \\[2mm]
    ={} & \quad\textup{by def. of $t_{X, Y\times Z}$ and by
      \eqref{eq:circ-pair}}
    \\[2mm]
    & \<\< 0\circ (\id_X\times t_{Y,Z})\circ a_{X,Y,TZ}, \pi_1\circ
    \pi_2\circ (\id_X\times t_{Y,Z})\circ a_{X,Y,TZ}\>,
    \\
    &\hphantom{\<} \<\pi_1\circ (\id_X\times t_{Y,Z})\circ a_{X,Y,TZ},
    \pi_2\circ \pi_2\circ (\id_X\times t_{Y,Z})\circ a_{X,Y,TZ}\>\>
    \displaybreak[0]
    \\[2mm]
    ={} & \quad\textup{by \eqref{eq:times-def} and because $\C$ is
      left-additive}
    \\[2mm]
    & \< \<0, \pi_1\circ t_{Y,Z}\circ \pi_2\circ a_{X,Y,TZ}\>,
    \<\pi_1\circ a_{X,Y,TZ}, \pi_2\circ t_{Y,Z}\circ \pi_2\circ
    a_{X,Y,TZ}\>\>
    \displaybreak[0]
    \\[2mm]
    ={} & \quad\textup{by def. of $t_{Y,Z}$ and $a_{X,Y,TZ}$, and by
      \eqref{eq:proj}}
    \\[2mm]
    & \<\< 0, \<0, \pi_1\circ\pi_2\> \circ \<\pi_2\circ\pi_1,
    \pi_2\>\>, \<\pi_1\circ\pi_1, \<\pi_1, \pi_2\circ\pi_2\>\circ
    \<\pi_2\circ\pi_1, \pi_2\>\>\>
    \displaybreak[0]
    \\[2mm]
    ={} & \quad\textup{by \eqref{eq:circ-pair}}
    \\[2mm]
    & \<\<0, \<0 \circ \<\pi_2\circ\pi_1, \pi_2\>,
    \pi_1\circ\pi_2\circ\<\pi_2\circ\pi_1, \pi_2\>\>\>,
    \\
    & \hphantom{\<} \<\pi_1\circ\pi_1, \<\pi_1\circ\<\pi_2\circ\pi_1,
    \pi_2\>, \pi_2\circ\pi_2\circ\<\pi_2\circ\pi_1, \pi_2\>\>\>\>
    \displaybreak[0]
    \\[2mm]
    ={} & \quad\textup{by \eqref{eq:proj} and because $\C$ is
      left-additive}
    \\[2mm]
    & \<\<0, \<0, \pi_1\circ\pi_2\>\>, \<\<\pi_1\circ\pi_1, \<
    \pi_2\circ\pi_1, \pi_2\circ\pi_2\>\>\>,
  \end{align*}
  which coincides with the expression for $T(a_{X,Y,Z})\circ
  t_{X\times Y,Z}$.  The lemma is proven.
\end{proof}

\begin{lemma}\label{lem-strength-unit}
  $t$ makes diagram \eqref{eq:strength-unit} commute.
\end{lemma}

\begin{proof}
  We have:
  \begin{align*}
    & t_{X,Y}\circ (\id_X\times\eta_Y)
    \\[2mm]
    ={} & \quad\textup{by def. of $t_{X,Y}$}
    \\[2mm]
    &\<\<0, \pi_1\circ\pi_2\>, \<\pi_1, \pi_2\circ\pi_2\>\> \circ
    (\id_X\times\eta_Y)
    \displaybreak[0]
    \\[2mm]
    ={} & \quad\textup{by \eqref{eq:circ-pair}}
    \\[2mm]
    & \<\<0 \circ (\id_X\times\eta_Y), \pi_1\circ\pi_2\circ
    (\id_X\times\eta_Y)\>, \<\pi_1\circ(\id_X\times\eta_Y),
    \pi_2\circ\pi_2\circ(\id_X\times\eta_Y)\>\>
    \displaybreak[0]
    \\[2mm]
    ={} & \quad\textup{by \eqref{eq:times-def} and because $\C$ is
      left-additive}
    \\[2mm]
    & \<\< 0, \pi_1\circ\eta_Y\circ\pi_2\>, \<\pi_1,
    \pi_2\circ\eta_Y\circ\pi_2\>\>
    \displaybreak[0]
    \\[2mm]
    ={} & \quad\textup{by def. of $\eta_Y$, by \eqref{eq:proj}, and
      because $\C$ is left-additive}
    \\[2mm]
    & \<\<0, 0\>, \<\pi_1, \pi_2\>\>
    \displaybreak[0]
    \\[2mm]
    ={} & \quad\textup{by \eqref{eq:id-pair} and because pairing is
      additive}
    \\[2mm]
    & \<0, \id\>
    \displaybreak[0]
    \\[2mm]
    ={} & \quad\textup{by def. of $\eta_{X\times Y}$}
    \\[2mm]
    & \eta_{X\times Y}.
  \end{align*}
  The lemma is proven.
\end{proof}

\begin{lemma}\label{lem-strength-mult}
  $t$ makes diagram \eqref{eq:strength-mult} commute.
\end{lemma}

\begin{proof}
  We have, on the one hand:
  \begin{align*}
    & \mu_{X\times Y} \circ T(t_{X,Y}) \circ t_{X,TY}
    \\[2mm]
    ={} & \quad\textup{by def. of $\mu_{X\times Y}$ and by
      \lemref{lemma-T-additive-morphisms} because $t_{X,Y}$ is
      linear}
    \\[2mm]
    & \<\pi_2\circ\pi_1 + \pi_1\circ\pi_2, \pi_2\circ\pi_2\>\circ
    (t_{X,Y}\times t_{X,Y})\circ t_{X,TY}
    \\[2mm]
    ={} & \quad\textup{by \eqref{eq:circ-pair} and because $\C$ is
      left-additive}
    \\[2mm]
    & \<\pi_2\circ\pi_1\circ(t_{X,Y}\times t_{X,Y})\circ t_{X,TY} +
    \pi_1\circ\pi_2\circ(t_{X,Y}\times t_{X,Y})\circ t_{X,TY},
    \\
    &\hphantom{\<}\pi_2\circ\pi_2\circ (t_{X,Y}\times t_{X,Y})\circ
    t_{X,TY}\>
    \displaybreak[0]
    \\[2mm]
    ={} & \quad\textup{by \eqref{eq:times-def}}
    \\[2mm]
    & \<\pi_2\circ t_{X,Y}\circ\pi_1\circ t_{X,TY} + \pi_1\circ
    t_{X,Y}\circ\pi_2\circ t_{X,TY}, \pi_2\circ t_{X,Y}\circ\pi_2\circ
    t_{X,TY}\>
    \displaybreak[0]
    \\[2mm]
    ={} & \quad\textup{by def. of $t$ and by \eqref{eq:proj}}
    \\[2mm]
    & \<\<\pi_1, \pi_2\circ\pi_2\> \circ \<0, \pi_1\circ\pi_2\> + \<0,
    \pi_1\circ\pi_2\> \circ \<\pi_1, \pi_2\circ\pi_2\>, \<\pi_1,
    \pi_2\circ \pi_2\> \circ \<\pi_1, \pi_2\circ\pi_2\>\>
    \displaybreak[0]
    \\[2mm]
    ={} & \quad\textup{by \eqref{eq:circ-pair}}
    \\[2mm]
    & \<\< \pi_1\circ\<0, \pi_1\circ\pi_2\>, \pi_2\circ\pi_2\circ\<0,
    \pi_1\circ\pi_2\>\> + \<0\circ\<\pi_1, \pi_2\circ\pi_2\>,
    \pi_1\circ\pi_2\circ\<\pi_1, \pi_2\circ\pi_2\>\>,
    \\
    & \hphantom{\<} \<\pi_1\circ\<\pi_1, \pi_2\circ\pi_2\>,
    \pi_2\circ\pi_2\circ\<\pi_1, \pi_2\circ\pi_2\>\>\>
    \displaybreak[0]
    \\[2mm]
    ={} & \quad\textup{by \eqref{eq:proj} and because $\C$ is
      left-additive}
    \\[2mm]
    & \<\<0, \pi_2\circ\pi_1\circ\pi_2\>
      + \<0, \pi_1\circ\pi_2\circ\pi_2\>,
        \<\pi_1, \pi_2\circ\pi_2\circ\pi_2\>\>
    \intertext{%
      On the other hand:
    }%
    & t_{X,Y}\circ(\id_X\times \mu_Y)
    \\[2mm]
    ={} & \quad\textup{by def. of $t_{X,Y}$ and \eqref{eq:circ-pair}}
    \\[2mm]
    & \<\<0 \circ (\id_X\times\mu_Y), \pi_1\circ\pi_2 \circ
    (\id_X\times\mu_Y)\>, \<\pi_1 \circ (\id_X\times\mu_Y),
    \pi_2\circ\pi_2 \circ (\id_X\times\mu_Y)\>
    \displaybreak[0]
    \\[2mm]
    ={} & \quad\textup{by \eqref{eq:times-def} and because $\C$ is
      left-additive}
    \\[2mm]
    & \<\<0, \pi_1\circ\mu_Y\circ\pi_2\>, \<\pi_1,
    \pi_2\circ\mu_Y\circ\pi_2\>\>
    \displaybreak[0]
    \\[2mm]
    ={} & \quad\textup{by def. of $\mu_Y$ and \eqref{eq:proj}}
    \\[2mm]
    & \<\<0, (\pi_2\circ\pi_1+\pi_1\circ\pi_2)\circ\pi_2\>, \<\pi_1,
    \pi_2\circ\pi_2\circ\pi_2\>\>
    \\[2mm]
    ={} & \quad\textup{because $\C$ is left-additive and pairing is
      additive}
    \\[2mm]
    & \<\<0, \pi_2\circ\pi_1\circ\pi_2\> + \<0,
    \pi_1\circ\pi_2\circ\pi_2\>, \<\pi_1,
    \pi_2\circ\pi_2\circ\pi_2\>\>.
  \end{align*}
  The obtained expressions are identical, hence the assertion.
\end{proof}

\begin{theorem}
  The natural transformation $t$ defined by \eqref{eq:strength-def} is
  a strength for the monad $(T, \eta, \mu)$.
\end{theorem}

\begin{proof}
  Follows from Lemmas~\ref{lem-strength-terminal},
  \ref{lem-strength-assoc}, \ref{lem-strength-unit}, and
  \ref{lem-strength-mult}.
\end{proof}

The tensorial strength $t$ is also called the right tensorial
strength.  Using the symmetry $c$ of $\C$, we may also define the left
tensorial strength by
\begin{equation}
  \label{eq:strength'-def}
  t'_{X,Y} = \left[TX\times Y\xrightarrow{c_{TX,Y}}Y\times
    TX\xrightarrow{t_{Y,X}}T(Y\times
    X)\xrightarrow{T(c_{Y,X})}T(X\times Y)\right].
\end{equation}

\begin{lemma}\label{lemma-strength'-expr}
  $t'_{X,Y} = \<\<\pi_1\circ\pi_1, 0\>, \<\pi_2\circ\pi_1, \pi_2\>\> =
  \<\pi_1\times 0, \pi_2\times\id_Y\>$.
\end{lemma}

\begin{proof}
  We have:
  \begin{alignat*}{3}
    T(c)\circ t\circ c
    &= (c\times c)\circ t\circ c
    & \qquad & \textup{by \lemref{lemma-T-additive-morphisms} because
      $c$ is linear}
    \displaybreak[0]
    \\
    &= (c\times c) \circ \<\<0, \pi_1\circ\pi_2\>,
       \<\pi_1, \pi_2\circ\pi_2\>\>\circ c
    & \qquad & \textup{by def. of $t$}
    \displaybreak[0]
    \\
    &= \<c\circ\<0, \pi_1\circ\pi_2\>\circ c, c\circ\<\pi_1,
    \pi_2\circ\pi_2\>\circ c\>
    & \qquad & \textup{by \eqref{eq:circ-pair} and
      \eqref{eq:times-circ-pair}}
    \displaybreak[0]
    \\
    &= \<\<\pi_1\circ\pi_2, 0\>\circ c,
         \<\pi_2\circ\pi_2, \pi_1\> \circ c\>
    & \qquad & \textup{by \eqref{eq:sym-pair}}
    \displaybreak[0]
    \\
    &= \<\<\pi_1\circ\pi_2\circ c, 0\circ c\>,
         \<\pi_2\circ\pi_2\circ c, \pi_1\circ c\>\>
    & \qquad & \textup{by \eqref{eq:circ-pair}}
    \displaybreak[0]
    \\
    &= \<\<\pi_1\circ\pi_1, 0\>, \<\pi_2\circ\pi_1, \pi_2\>\>
    & \qquad & \textup{by def. of $c$, by \eqref{eq:proj},}
    \\
    & & \qquad & \textup{and because $\C$ is left-additive}
    \displaybreak[0]
    \\
    &= \<\pi_1\times 0, \pi_2\times\id\>
    & \qquad & \textup{by \eqref{eq:times-def}},
  \end{alignat*}
  as asserted.
\end{proof}

\subsection{The monoidal structure of $T$}
\label{sec-monoidal-structure}

Because the functor $T$ is part of a strong monad, by
\citep[Theorem~2.1]{MR0260825} $T$ becomes a monoidal functor $(T,
\psi, \psi^0):\C\to\C$ if we put $\psi_{X,Y}$ equal to the composite
\begin{equation}
  \label{eq:psi-def}
  \psi_{X,Y} = \left[TX\times TY\xrightarrow{t'_{X,TY}}T(X\times
    TY)\xrightarrow{T(t_{X,Y})}TT(X\times Y)\xrightarrow{\mu_{X\times
        Y}} T(X\times Y)\right]
\end{equation}
and by putting $\psi^0=\eta_\unito:\unito\to T\unito$.  The definition
of $\psi_{X,Y}$ is asymmetric, and indeed there is also a morphism
\begin{equation}
  \label{eq:psi-tilde-def}
  \tilde\psi_{X,Y} = \left[TX\times TY\xrightarrow{t_{TX,Y}}T(TX\times
  Y)\xrightarrow{T(t'_{X,Y})}TT(X\times Y)\xrightarrow{\mu_{X\times
      Y}}T(X\times Y)\right]
\end{equation}
that also makes $T$ into a monoidal functor.  Strong monads for which
$\psi$ and $\tilde\psi$ agree are called \emph{commutative}
\citep[Definition~3.1]{MR0260825}.  Let us prove that the tangent
bundle monad is commutative by computing the morphisms $\psi$ and
$\tilde\psi$ explicitly and showing that they are equal.

\begin{lemma}\label{lemma-psi-is-shuffle}
  $\psi_{X,Y} = \<\<\pi_1\circ\pi_1, \pi_1\circ\pi_2\>,
  \<\pi_2\circ\pi_1, \pi_2\circ\pi_2\>\>: (X\times X)\times (Y\times
  Y)\to (X\times Y)\times (X\times Y)$.
\end{lemma}

\begin{proof}
  Taking into account that $T(t_{X,Y}) = t_{X,Y}\times t_{X,Y}$ by
  \lemref{lemma-T-additive-morphisms} because $t_{X,Y}$ is linear,
  we have:
  \begin{align*}
    & \mu\circ T(t) \circ t'
    \\[2mm]
    ={} & \quad\textup{by def. of $\mu$}
    \\[2mm]
    & \<\pi_2\circ\pi_1 + \pi_1\circ\pi_2, \pi_2\circ\pi_2\> \circ
    T(t) \circ t'
    \displaybreak[0]
    \\[2mm]
    ={} & \quad\textup{by \eqref{eq:circ-pair} and because $\C$ is
      left-additive}
    \\[2mm]
    & \<\pi_2\circ\pi_1\circ T(t)\circ t' + \pi_1\circ\pi_2\circ
    T(t)\circ t', \pi_2\circ\pi_2\circ T(t) \circ t'\>
    \displaybreak[0]
    \\[2mm]
    ={} & \quad\textup{by \eqref{eq:times-def} and because $T(t) =
      t\times t$}
    \\[2mm]
    & \<\pi_2\circ t\circ\pi_1\circ t' + \pi_1\circ t\circ\pi_2\circ
    t', \pi_2\circ t\circ\pi_2\circ t'\>
    \displaybreak[0]
    \\[2mm]
    ={} & \quad\textup{by def. of $t$ and $t'$ and by \eqref{eq:proj}}
    \\[2mm]
    & \<\<\pi_1,\pi_2\circ\pi_2\> \circ\<\pi_1\circ\pi_1, 0\> + \<0,
    \pi_1\circ\pi_2\>\circ\<\pi_2\circ\pi_1, \pi_2\>, \<\pi_1,
    \pi_2\circ\pi_2\>\circ\<\pi_2\circ\pi_1, \pi_2\>\>
    \displaybreak[0]
    \\[2mm]
    ={} & \quad\textup{by \eqref{eq:circ-pair}}
    \\[2mm]
    & \<\<\pi_1\circ\<\pi_1\circ\pi_1, 0\>,
    \pi_2\circ\pi_2\circ\<\pi_1\circ\pi_1, 0\>\> +
    \<0\circ\<\pi_2\circ\pi_1, \pi_2\>,
    \pi_1\circ\pi_2\circ\<\pi_2\circ\pi_1, \pi_2\>\>,
    \\
    &\hphantom{\<} \<\pi_1\circ\<\pi_2\circ\pi_1, \pi_2\>,
    \pi_2\circ\pi_2\circ\<\pi_2\circ\pi_1, \pi_2\>\>\>
    \displaybreak[0]
    \\[2mm]
    ={} & \quad\textup{by \eqref{eq:proj} and because $\C$ is
      left-additive and $\pi_2$ is additive}
    \\[2mm]
    & \<\<\pi_1\circ\pi_1, 0\> + \<0, \pi_1\circ\pi_2\>,
    \<\pi_2\circ\pi_1, \pi_2\circ\pi_2\>\>
    \displaybreak[0]
    \\[2mm]
    ={} & \quad\textup{because pairing is additive}
    \\[2mm]
    & \<\<\pi_1\circ\pi_1, \pi_1\circ\pi_2\>, \<\pi_2\circ\pi_1,
    \pi_2\circ\pi_2\>\>,
  \end{align*}
  as asserted.
\end{proof}

\begin{lemma}\label{lemma-tilde-psi-is-shuffle}
  $\tilde\psi_{X,Y}=\<\<\pi_1\circ\pi_1, \pi_1\circ\pi_2\>,
  \<\pi_2\circ\pi_1, \pi_2\circ\pi_2\>\>$.
\end{lemma}

\begin{proof}
  The proof is similar to the proof of
  \lemref{lemma-psi-is-shuffle}.  We have:
  \begin{align*}
    & \mu\circ T(t')\circ t
    \\[2mm]
    ={} & \quad\textup{by def. of $\mu$ and because $T(t')=t'\times
      t'$, see \lemref{lemma-psi-is-shuffle}}
    \\[2mm]
    & \<\pi_2\circ t'\circ\pi_1\circ t + \pi_1\circ t'\circ\pi_2\circ
    t, \pi_2\circ t'\circ\pi_2\circ t\>
    \displaybreak[0]
    \\[2mm]
    ={} & \quad\textup{by def. of $t$ and $t'$ and by \eqref{eq:proj}}
    \\[2mm]
    & \<\<\pi_2\circ\pi_1, \pi_2\> \circ \<0, \pi_1\circ\pi_2\> +
    \<\pi_1,\pi_2\circ\pi_2\> \circ \<\pi_1\circ\pi_1, 0\>,
    \<\pi_2\circ\pi_1, \pi_2\> \circ \<\pi_1, \pi_2\circ \pi_2\>\>
    \displaybreak[0]
    \\[2mm]
    ={} & \quad\textup{by \eqref{eq:circ-pair}}
    \\[2mm]
    & \<\<\pi_2\circ\pi_1\circ\<0, \pi_1\circ\pi_2\>, \pi_2\circ\<0,
    \pi_1\circ\pi_2\>\> + \<\pi_1\circ\<\pi_1\circ\pi_1, 0\>,
    \pi_2\circ\pi_2\circ\<\pi_1\circ\pi_1, 0\>\>,
    \\
    & \hphantom{\<} \<\pi_2\circ\pi_1\circ\<\pi_1, \pi_2\circ\pi_2\>,
    \pi_2\circ\<\pi_1, \pi_2\circ\pi_2\>\>\>
    \displaybreak[0]
    \\[2mm]
    ={} & \quad\textup{by \eqref{eq:proj} and because $\C$ is
      left-additive and $\pi_2$ is additive}
    \\[2mm]
    & \<\<0, \pi_1\circ\pi_2\> + \<\pi_1\circ\pi_1, 0\>,
    \<\pi_2\circ\pi_1, \pi_2\circ\pi_2\>\>
    \displaybreak[0]
    \\[2mm]
    ={} & \quad\textup{because pairing is additive}
    \\[2mm]
    & \<\<\pi_1\circ\pi_1, \pi_1\circ\pi_2\>, \<\pi_2\circ\pi_1,
    \pi_2\circ\pi_2\>\>.
  \end{align*}
  The lemma is proven.
\end{proof}

\begin{theorem}
  The monad $(T, \eta, \mu)$ is commutative.
\end{theorem}

\begin{proof}
  Follows from Lemmas~\ref{lemma-psi-is-shuffle} and
  \ref{lemma-tilde-psi-is-shuffle}.
\end{proof}

\begin{remark}\label{remark-psi-iso}
  Note that by \eqref{eq:times-def}, $\psi=\tilde\psi$ can also be
  written as $\<\pi_1\times\pi_1, \pi_2\times\pi_2\>$.
  Set-theoretically it is given by $\psi((x', x), (y', y)) = ((x',
  y'), (x, y))$.  It is also straightforward that $\psi$ is invertible
  with the inverse $\<T(\pi_1), T(\pi_2)\>=\<\pi_1\times\pi_1,
  \pi_2\times\pi_2\>$.  In particular, $T$ preserves products.
\end{remark}

\begin{lemma}\label{lemma-psi-pair}
  Let $f:Z\to X$, $g:Z\to Y$ be morphisms in $\C$.  Then $T(\<f, g\>)
  = \psi\circ\<T(f), T(g)\>$.
\end{lemma}

\begin{proof}
  Because $\<f, g\> = (f\times g)\circ\Delta$, where
  $\Delta=\<\id_Z,\id_Z\>:Z\to Z\times Z$ is the diagonal morphism,
  the claim follows from the naturality of $\psi$ and from the
  equation $\psi\circ\Delta = T(\Delta)$, which is proved as follows:
  \begin{alignat*}{3}
    \psi \circ \Delta
    &= \<\pi_1\times\pi_1, \pi_2\times\pi_2\>\circ\<\id,\id\>
    & \qquad & \textup{by \remref{remark-psi-iso} and def. of
      $\Delta$}
    \displaybreak[0]
    \\
    &= \<(\pi_1\times\pi_1)\circ\<\id, \id\>,
         (\pi_2\times\pi_2)\circ\<\id, \id\>\>
    & \qquad & \textup{by \eqref{eq:circ-pair}}
    \displaybreak[0]
    \\
    &= \< \<\pi_1, \pi_1\>, \<\pi_2, \pi_2\>\>
    & \qquad & \textup{by \eqref{eq:times-circ-pair}},
    \\
    \displaybreak[0]
    \\
    T(\Delta)
    & = \<D(\<\id,\id\>), \<\id, \id\>\circ\pi_2\>
    & \qquad & \textup{by def. of $T$ and $\Delta$}
    \displaybreak[0]
    \\
    & = \<\<\pi_1, \pi_1\>, \<\pi_2, \pi_2\>\>
    & \qquad & \textup{by D4, D3, and \eqref{eq:circ-pair}}.
  \end{alignat*}
  The lemma is proven.
\end{proof}

\subsection{The distributive law of $T$ over itself}
\label{sec-distributive-law}

We are going to prove that the distributivity isomorphism
\[
\sigma = \<\<\pi_1\circ \pi_1, \pi_1\circ\pi_2\>, \<\pi_2\circ\pi_1,
\pi_2\circ\pi_2\>\>: (X\times X)\times (X\times X)\xrightarrow\sim
(X\times X)\times (X\times X)
\]
defines a distributive law of the monad $T$ over itself.

\begin{lemma}
  $\sigma$ is a natural transformation $TT\to TT$.
\end{lemma}

\begin{proof}
  We must show that for any morphism $f$ holds $T(T(f)) \circ \sigma =
  \sigma \circ T(T(f))$.  We have:
  \begin{align*}
    & T(T(f)) \circ \sigma
    \\[2mm]
    ={} & \quad\textup{by def. of $T$ and by \eqref{eq:circ-pair}}
    \\[2mm]
    & \<D(T(f)) \circ \sigma, T(f) \circ \pi_2 \circ \sigma\>
    \displaybreak[0]
    \\[2mm]
    ={} & \quad\textup{by def. of $T$, D4, D5, and D3}
    \\[2mm]
    &\<\<D(D(f)), D(f)\circ \<\pi_2\circ\pi_1,
    \pi_2\circ\pi_2\>\>\circ \sigma, \<D(f),
    f\circ\pi_2\>\circ\pi_2\circ\sigma\>
    \displaybreak[0]
    \\[2mm]
    ={} & \quad\textup{by \eqref{eq:circ-pair}}
    \\[2mm]
    & \<\<D(D(f))\circ\sigma, D(f)\circ \<\pi_2\circ\pi_1\circ\sigma,
    \pi_2\circ\pi_2\circ\sigma\>\>, \<D(f)\circ\pi_2\circ\sigma,
    f\circ\pi_2\circ\pi_2\circ\sigma\>\>
    \displaybreak[0]
    \\[2mm]
    ={} & \quad\textup{by def. of $\sigma$ and \eqref{eq:proj}}
    \\[2mm]
    & \<\<D(D(f))\circ\sigma, D(f)\circ \<\pi_1\circ\pi_2,
    \pi_2\circ\pi_2\>\>, \<D(f)\circ \<\pi_2\circ\pi_1,
    \pi_2\circ\pi_2\>, f\circ\pi_2\circ\pi_2\>\>
    \displaybreak[0]
    \\[2mm]
    ={} & \quad\textup{by \eqref{eq:circ-pair}, \eqref{eq:id-pair},
      D5, and D3}
    \\[2mm]
    & \<\<D(D(f))\circ \sigma, D(f)\circ\pi_2\>, \<D(f\circ\pi_2),
    f\circ\pi_2\circ\pi_2\>\>
    \displaybreak[0]
    \\[2mm]
    ={} & \quad\textup{by \lemref{lemma-D-interchange}}
    \\[2mm]
    & \<\<D(D(f)), D(f)\circ\pi_2\>, \<D(f\circ\pi_2),
    f\circ\pi_2\circ\pi_2\>\>
    \displaybreak[0]
    \\[2mm]
    ={} & \quad\textup{by \eqref{eq:sigma-pair}}
    \\[2mm]
    &\sigma \circ \<\<D(D(f)), D(f\circ\pi_2)\>, \<D(f)\circ\pi_2,
    f\circ\pi_2\circ\pi_2\>\>
    \displaybreak[0]
    \\[2mm]
    ={} & \quad\textup{by def. of $T$, \eqref{eq:circ-pair}, and D4}
    \\[2mm]
    & \sigma \circ \<D(T(f)), T(f) \circ\pi_2\>
    \displaybreak[0]
    \\[2mm]
    ={} & \quad\textup{by def. of $T$}
    \\[2mm]
    & \sigma\circ T(T(f)).
  \end{align*}
  The lemma is proven.
\end{proof}

We recall from \citet{MR0241502} that a natural transformation
$\sigma : T^2\to T^2$ is a distributive law of the monad $T$ over
itself if $\sigma$ makes the following diagrams commute:
\[
\begin{xy}
  (-20, 9)*+{T^3}="1";
  (0, 9)*+{T^3}="2";
  (20, 9)*+{T^3}="3";
  (-20, -9)*+{T^2}="4";
  (20, -9)*+{T^2}="5";
  {\ar@{->}^-{\sigma T} "1";"2"};
  {\ar@{->}^-{T\sigma} "2";"3"};
  {\ar@{->}_-{T\mu} "1";"4"};
  {\ar@{->}^-{\sigma} "4";"5"};
  {\ar@{->}^-{\mu T} "3";"5"};
  (0, -18)*+{T}="6";
  (-10, -32)*+{T^2}="7";
  (10, -32)*+{T^2}="8";
  {\ar@{->}_-{T\eta} "6";"7"};
  {\ar@{->}^-{\eta T} "6";"8"};
  {\ar@{->}^-{\sigma} "7";"8"};
\end{xy}
\qquad
\begin{xy}
  (-20, 9)*+{T^3}="1";
  (0, 9)*+{T^3}="2";
  (20, 9)*+{T^3}="3";
  (-20, -9)*+{T^2}="4";
  (20, -9)*+{T^2}="5";
  {\ar@{->}^-{T\sigma} "1";"2"};
  {\ar@{->}^-{\sigma T} "2";"3"};
  {\ar@{->}_-{\mu T} "1";"4"};
  {\ar@{->}^-{\sigma} "4";"5"};
  {\ar@{->}^-{T\mu} "3";"5"};
  (0, -18)*+{T}="6";
  (-10, -32)*+{T^2}="7";
  (10, -32)*+{T^2}="8";
  {\ar@{->}_-{\eta T} "6";"7"};
  {\ar@{->}^-{T\eta} "6";"8"};
  {\ar@{->}^-{\sigma} "7";"8"};
\end{xy}
\]

\begin{theorem}
  $\sigma$ is a distributive law of the monad $T$ over itself.
\end{theorem}

\begin{proof}
  We prove only the commutativity of the left pentagon and the right
  triangle.  The proofs of the commutativity of the other two diagrams
  are similar.  Let us check that the triangle commutes, i.e.,
  $\sigma\circ\eta T = T(\eta)$.  We have:
  \begin{alignat*}{3}
    \sigma \circ \eta T
    &= \sigma \circ \<0, \id\>
    & \qquad & \textup{by def. of $\eta$}
    \displaybreak[0]
    \\
    &= \sigma \circ \<\<0, 0\>, \<\pi_1, \pi_2\>\>
    & \qquad & \textup{by \eqref{eq:id-pair} and because pairing
      is additive}
    \displaybreak[0]
    \\
    &= \<\<0, \pi_1\>, \<0, \pi_2\>\>
    & \qquad & \textup{by \eqref{eq:sigma-pair}},
    \\
    \displaybreak[0]
    \\
    T(\eta)
    & = \eta\times\eta
    & \qquad & \textup{by \lemref{lemma-T-additive-morphisms}
      because $\eta$ is linear}
    \displaybreak[0]
    \\
    & = \<\eta\circ\pi_1, \eta\circ\pi_2\>
    & \qquad & \textup{by \eqref{eq:times-def}}
    \displaybreak[0]
    \\
    & = \<\<0, \pi_1\>, \<0, \pi_2\>\>
    & \qquad & \textup{by \eqref{eq:circ-pair} and because $\C$
      is left-additive}.
  \end{alignat*}
  Let us prove the commutativity of the pentagon.  We have, on the one
  hand:
  \begin{align*}
    & \sigma\circ T(\mu)
    \\[2mm]
    ={} & \quad\textup{by \lemref{lemma-T-additive-morphisms}
      because $\mu$ is linear}
    \\[2mm]
    & \sigma \circ (\mu\times \mu)
    \displaybreak[0]
    \\[2mm]
    ={} & \quad\textup{by def. of $\mu$}
    \\[2mm]
    & \sigma\circ \<\pi_2\circ\pi_1 + \pi_1\circ\pi_2,
    \pi_2\circ\pi_2\> \times \<\pi_2\circ\pi_1 + \pi_1\circ\pi_2,
    \pi_2\circ\pi_2\>
    \displaybreak[0]
    \\[2mm]
    ={} & \quad\textup{by \eqref{eq:sigma-times}}
    \\[2mm]
    & \<(\pi_2\circ\pi_1+\pi_1\circ\pi_2)\times
    (\pi_2\circ\pi_1+\pi_1\circ\pi_2), (\pi_2\circ\pi_2) \times
    (\pi_2\circ\pi_2)\>
    \displaybreak[0]
    \\[2mm]
    ={} & \quad\textup{because product is additive}
    \\[2mm]
    & \<(\pi_2\circ\pi_1)\times (\pi_2\circ\pi_1) + (\pi_1\circ\pi_2)
    \times (\pi_1\circ\pi_2), (\pi_2\circ\pi_2)\times
    (\pi_2\circ\pi_2)\>
    \displaybreak[0]
    \\[2mm]
    \intertext{%
      On the other hand:
    }%
    & \mu\circ T(\sigma) \circ \sigma
    \\[2mm]
    ={} & \quad\textup{by def. of $\mu$ and \eqref{eq:circ-pair}}
    \\[2mm]
    & \<\pi_2\circ\pi_1\circ T(\sigma)\circ \sigma +
    \pi_1\circ\pi_2\circ T(\sigma)\circ \sigma, \pi_2\circ\pi_2\circ
    T(\sigma)\circ \sigma\>
    \displaybreak[0]
    \\[2mm]
    ={} & \quad\textup{by \lemref{lemma-T-additive-morphisms}
      because $\sigma$ is linear}
    \\[2mm]
    & \<\pi_2\circ\pi_1\circ (\sigma\times\sigma)\circ \sigma +
    \pi_1\circ\pi_2\circ (\sigma\times\sigma)\circ \sigma,
    \pi_2\circ\pi_2\circ (\sigma\times\sigma)\circ \sigma\>
    \displaybreak[0]
    \\[2mm]
    ={} & \quad\textup{by \eqref{eq:times-def}}
    \\[2mm]
    & \<\pi_2\circ\sigma\circ\pi_1\circ\sigma +
    \pi_1\circ\sigma\circ\pi_2\circ\sigma,
    \pi_2\circ\sigma\circ\pi_2\circ\sigma\>
    \displaybreak[0]
    \\[2mm]
    ={} & \quad\textup{by def. of $\sigma$ and \eqref{eq:proj}}
    \\[2mm]
    & \<\<\pi_2\circ\pi_1, \pi_2\circ\pi_2\>\circ\<\pi_1\circ\pi_1,
    \pi_1\circ\pi_2\> + \<\pi_1\circ\pi_1,
    \pi_1\circ\pi_2\>\circ\<\pi_2\circ\pi_1, \pi_2\circ\pi_2\>,
    \\
    & \hphantom{\<}\<\pi_2\circ\pi_1, \pi_2\circ\pi_2\>
    \circ\<\pi_2\circ\pi_1, \pi_2\circ\pi_2\>\>
    \displaybreak[0]
    \\[2mm]
    ={} & \quad\textup{by \eqref{eq:circ-pair} and \eqref{eq:proj}}
    \\[2mm]
    & \<\<\pi_2\circ\pi_1\circ\pi_1, \pi_2\circ\pi_1\circ\pi_2\> +
    \<\pi_1\circ\pi_2\circ\pi_1, \pi_1\circ\pi_2\circ\pi_2\>,
    \<\pi_2\circ\pi_2\circ\pi_1, \pi_2\circ\pi_2\circ\pi_2\>\>
    \displaybreak[0]
    \\[2mm]
    ={} & \quad\textup{by \eqref{eq:times-def}}
    \\[2mm]
    & \<(\pi_2\circ\pi_1)\times (\pi_2\circ\pi_1) + (\pi_1\circ\pi_2)
    \times (\pi_1\circ\pi_2), (\pi_2\circ\pi_2)\times
    (\pi_2\circ\pi_2)\>.
  \end{align*}
  The obtained expressions are identical, hence the assertion.  The
  theorem is proven.
\end{proof}

\begin{proposition}\label{prop-sigma-compat-strength}
  The diagram
  \[
  \begin{xy}
    (-30,9)*+{TX\times TY}="1";
    (30,9)*+{T(TX\times Y)}="2";
    (-30,-9)*+{T(X\times TY)}="3";
    (0, -9)*+{TT(X\times Y)}="4";
    (30, -9)*+{TT(X\times Y)}="5";
    {\ar@{->}^-{t} "1";"2"};
    {\ar@{->}^-{T(t')} "2";"5"};
    {\ar@{->}_-{t'} "1";"3"};
    {\ar@{->}^-{T(t)} "3";"4"};
    {\ar@{->}^-{\sigma} "4";"5"};
  \end{xy}
  \]
  commutes.
\end{proposition}

\begin{proof}
  We have:
  \begin{align*}
    & T(t')\circ t
    \\[2mm]
    ={} & \quad\textup{by \lemref{lemma-T-additive-morphisms}
      because $t'$ is linear}
    \\[2mm]
    & (t'\times t')\circ t
    \displaybreak[0]
    \\[2mm]
    ={} & \quad\textup{by def. of $t$}
    \\[2mm]
    & (t'\times t')\circ\<0\times\pi_1, \id\times\pi_2\>
    \displaybreak[0]
    \\[2mm]
    ={} & \quad\textup{by \eqref{eq:times-circ-pair}}
    \\[2mm]
    & \<t'\circ(0\times\pi_1), t'\circ(\id\times\pi_2)\>
    \displaybreak[0]
    \\[2mm]
    ={} & \quad\textup{by \lemref{lemma-strength'-expr}}
    \\[2mm]
    & \<\<\pi_1\times 0, \pi_2\times\id\>\circ(0\times\pi_1),
    \<\pi_1\times 0, \pi_2\times\id\>\circ(\id\times\pi_2)\>
    \displaybreak[0]
    \\[2mm]
    ={} & \quad\textup{by \eqref{eq:circ-pair}, functoriality of
      $\times$}
    \\[2mm]
    & \<\<\pi_1\circ 0\times 0\circ\pi_1, \pi_2\circ
    0\times\id\circ\pi_1\>, \<\pi_1\circ\id\times 0\circ\pi_2,
    \pi_2\circ\id\times \id\circ\pi_2\>\>
    \displaybreak[0]
    \\[2mm]
    ={} & \quad\textup{because $\C$ is left-additive and projections
      are additive}
    \\[2mm]
    & \<\<0\times 0, 0\times\pi_1\>, \<\pi_1\times 0,
    \pi_2\times\pi_2\>\>.
  \end{align*}
  Similarly, $T(t)\circ t' = \<\<0\times 0, \pi_1\times 0\>,
  \<0\times\pi_1, \pi_2\times\pi_2\>\>$, and the assertion follows by
  \eqref{eq:sigma-pair}.
\end{proof}

\subsection{The tangent bundle monad on a differential
  $\lambda$-category}
\label{sec-monad-on-diff-cat}

So far, we have only assumed that $\C$ is a cartesian differential
category.  From now on we suppose that $\C$ is a differential
$\lambda$-category.  Let us see what property of the functor $T$
condition~\eqref{eq:D-curry} translates into.  First, observe that
$\<\pi_1\times 0_X, \pi_2\times\id_X\>$ is precisely the tensorial
strength $t'$ by \lemref{lemma-strength'-expr}.  Therefore, equation
\eqref{eq:D-curry} can be written equivalently as
\begin{equation}
  \label{eq:D-curry-1}
  D(\curry(f)) = \curry(D(f)\circ t').
\end{equation}
Substituting $\uncurry(g)$ for $f$ and applying $\uncurry$ to both
sides of this equation, we conclude that condition~\eqref{eq:D-curry}
is equivalent to the following one: for each $g: Z\to X\Rightarrow Y$
holds
\begin{equation}
  \label{eq:D-uncurry}
  \uncurry (D(g)) = D(\uncurry(g))\circ t'.
\end{equation}
The left hand side is equal to
\begin{alignat*}{3}
  \uncurry(D(g))
  & = \ev\circ (D(g)\times \id_X)
  & \qquad & \textup{by def. of $\uncurry$}
  \\
  & = \ev\circ (\pi_1\times \id_X)\circ (T(g)\times \id_X)
  & \qquad & \textup{by def. of $T$, functoriality of $\times$,
    and \eqref{eq:proj}}.
\end{alignat*}
The right hand side of equation~\eqref{eq:D-uncurry} is equal to
\begin{alignat*}{3}
  D(\ev\circ(g\times \id_X))\circ t'
  &= D(\ev)\circ T(g\times\id_X) \circ t'
  & \qquad & \textup{by D5}
  \\
  &= D(\ev)\circ t'\circ (T(g)\times\id_X)
  & \qquad & \textup{by naturality of $t'$}.
\end{alignat*}
We conclude that equation~\eqref{eq:D-uncurry} is equivalent to
\begin{equation}
  \label{eq:D-uncurry-1}
  \ev\circ (\pi_1\times \id_X)\circ (T(g)\times \id_X)
  = D(\ev)\circ t'\circ (T(g)\times\id_X).
\end{equation}
Equation~\eqref{eq:D-uncurry-1} must hold for each $g:Z\to
X\Rightarrow Y$, in particular for $g = \id_{X\Rightarrow Y}$, in
which case it reduces to
\begin{equation}
  \label{eq:D-uncurry-2}
  \ev\circ (\pi_1\times \id_X) = D(\ev)\circ t'.
\end{equation}
Conversely, if equation~\eqref{eq:D-uncurry-2} holds, then by
precomposing both sides with $T(g)\times\id_X$, we find that
equation~\eqref{eq:D-uncurry-1} holds, too.

\begin{proposition}\label{prop-diff-lambda-cat}
  A cartesian closed differential category is a differential
  $\lambda$-category if and only if each evaluation morphism satisfies
  equation~\eqref{eq:D-uncurry-2}.
\end{proposition}

\begin{proposition}\label{prop-T-curry}
  Let $g: A\times B\to C$ be a morphism in $\C$.  Let $h=T(g)\circ
  t':TA\times B\to TC$.  Then $T(\curry(g)) = \<\curry(\pi_1\circ h),
  \curry(\pi_2\circ h)\>: TA\to T(B\Rightarrow C)$.
\end{proposition}

\begin{proof}
  We need to prove an equation between two morphisms into
  $T(B\Rightarrow C)=(B\Rightarrow C)\times(B\Rightarrow C)$.  This is
  equivalent to proving the two equations obtained by postcomposing
  the equation in question with the two projections.  That is, we need
  to prove the equations:
  \begin{align}
    \pi_1\circ T(\curry(g)) & = \curry (\pi_1\circ h),
    \label{eq:pi1-T-curry}\\
    \pi_2\circ T(\curry(g)) & = \curry (\pi_2\circ
    h).\label{eq:pi2-T-curry}
  \end{align}
  By the definition of $T$, the left hand side of
  \eqref{eq:pi1-T-curry} is equal to $D(\curry(g))$ and the right hand
  side is equal to $\curry(D(g)\circ t')$, hence equation
  \eqref{eq:pi1-T-curry} follows from \eqref{eq:D-curry-1}.  The left
  hand side of \eqref{eq:pi2-T-curry} is equal to
  $\curry(g)\circ\pi_2$, whereas the right hand side is equal to
  $\curry(g\circ\pi_2\circ t')$.  Note that $\pi_2\circ t' =
  \pi_2\times\id$ by \lemref{lemma-strength'-expr}, hence
  $\curry(g\circ\pi_2\circ t') = \curry(g\circ(\pi_2\times\id)) =
  \curry(g)\circ \pi_2$ by \eqref{eq:curry}.
\end{proof}

\subsection{The closed structure of $T$}
\label{sec-closed-structure}

\label{remark-closed-T}
The cartesian closed category $\C$ is a symmetric monoidal closed
category, and hence also a closed category of
\citet{MR0225841}.  By \citep[Proposition~4.3]{MR0225841}, the monoidal
functor $(T,\psi,\psi^0): \C\to\C$ gives rise to a closed functor
$(T,\hat\psi,\psi^0):\C\to\C$, where $\hat\psi=\hat\psi_{X,Y}:
T(X\Rightarrow Y)\to (TX\Rightarrow TY)$ is given by
$\hat\psi=\curry(T(\ev)\circ\psi)$.  We claim that
\begin{equation}
  \label{eq:T-ev-psi}
  T(\ev)\circ \psi = \<\ev\circ(\pi_1\times\pi_2) + D(\ev)\circ t\circ
  (\pi_2\times \id), \ev\circ(\pi_2\times\pi_2)\>.
\end{equation}
By the definition of $T$ and equation~\eqref{eq:circ-pair} we have
$T(\ev)\circ\psi=\<D(\ev)\circ\psi, \ev\circ\pi_2\circ\psi\>$.  The
morphism $\ev\circ\pi_2\circ\psi$ is equal to
$\ev\circ(\pi_2\times\pi_2)$ by \remref{remark-psi-iso}, so it remains
to show that
\begin{equation}
  \label{eq:D-ev-psi}
  D(\ev)\circ\psi = \ev\circ(\pi_1\times\pi_2) + D(\ev)\circ t\circ
  (\pi_2\times\id).
\end{equation}
Let us compute each summand in the right hand side separately.  We
have:
\begin{alignat*}{3}
  \ev\circ(\pi_1\times\pi_2)
  &= \ev\circ(\pi_1\times\id)\circ (\id\times\pi_2)
  & \qquad & \textup{by functoriality of $\times$}
  \displaybreak[0]
  \\
  &= D(\ev)\circ t'\circ (\id\times\pi_2)
  & \qquad & \textup{by \eqref{eq:D-uncurry-2}}
  \displaybreak[0]
  \\
  &= D(\ev)\circ \<\pi_1\times 0,
  \pi_2\times\id\>\circ(\id\times\pi_2)
  & \qquad & \textup{by \lemref{lemma-strength'-expr}}
  \displaybreak[0]
  \\
  & = D(\ev)\circ \<0\times\pi_1, \pi_2\times\pi_2\>
  & \qquad & \textup{by \eqref{eq:circ-pair}, functoriality
    of $\times$,}
  \\
  & & \qquad & \textup{and because $\C$ is left-additive},
  \\
  \displaybreak[0]
  \\
  D(\ev)\circ t\circ (\pi_2\times \id)
  & = D(\ev)\circ\<0\times\pi_1, \id\times\pi_2\>
  \circ(\pi_2\times\id)
  & \qquad & \textup{by def. of $t$}
  \displaybreak[0]
  \\
  & = D(\ev)\circ\<\pi_1\times 0, \pi_2\times\pi_2\>
  & \qquad & \textup{by \eqref{eq:circ-pair}, functoriality
    of $\times$,}
  \\
  & & \qquad & \textup{and because $\C$ is left-additive}.
\end{alignat*}
Therefore, by D2 and because $\times$ is additive, we have:
\begin{align*}
  \ev\circ(\pi_1\times\pi_2) + D(\ev)\circ t\circ
  (\pi_2\times \id) & = D(\ev)
  \circ\<0\times\pi_1, \pi_2\times\pi_2\> + D(\ev)
  \circ\<\pi_1\times 0, \pi_2\times\pi_2\>
  \\
  & = D(\ev) \circ \<\pi_1\times\pi_1, \pi_2\times\pi_2\>
  \\
  & = D(\ev) \circ \psi,
\end{align*}
proving \eqref{eq:D-ev-psi}.

\subsection{The enrichment of $T$}
\label{sec-enrichment}

\label{remark-underline-T}
We recall that the cartesian closed category $\C$ gives rise to a
category $\underline\C$ enriched in $\C$.  The objects $\underline\C$
are the objects of $\C$, and for each pair of object $X$ and $Y$ of
$\C$, $\underline\C(X,Y) = X\Rightarrow Y$.  The identity of an object
$X$ is the morphism $e_X=\curry(\ell^{-1}):\unito\to X\Rightarrow X$,
and the composition morphism $m_{X,Y,Z}: (X\Rightarrow Y)\times
(Y\Rightarrow Z)\to (X\Rightarrow Z)$ is given by $\curry
(\ev\circ(\id\times\ev)\circ a)$.  By \citep[Theorem~1.3]{MR0304456},
the functor $T:\C\to\C$, being equipped with the tensorial strength $t$, gives
rise to a $\C$-functor $\underline{T}:\underline\C\to\underline\C$
such that $\underline{T}X=TX$ and $\underline{T} =
\underline{T}_{X,Y}:(X\Rightarrow Y)\to (TX\Rightarrow TY)$ is given
by $\underline{T}=\curry(T(\ev)\circ t)$.  The definitions of $T$ and
$t$ imply that
\begin{equation}
  \label{eq:T-ev-t}
  T(\ev)\circ t = \<D(\ev)\circ t, \ev\circ(\id\times\pi_2)\>.
\end{equation}
Let us prove that the morphism $\underline{T}$ is linear. The proof
relies on the following criterion of linearity of curried morphisms.

\begin{lemma}\label{lemma-curry-linear}
  Let $f:Z\times X\to Y$ be a morphism in $\C$.  Then the currying
  $\curry(f):Z\to X\Rightarrow Y$ is a linear morphism if and only if
  $D(f)\circ t' = f \circ (\pi_1\times\id)$.
\end{lemma}

\begin{proof}
  By definition, $\curry(f)$ is linear if and only if $D(\curry(f)) =
  \curry(f)\circ\pi_1$.  Applying $\uncurry$ to both sides of the
  equation we obtain an equivalent equation
  $\uncurry(D(\curry(f)))=\uncurry(\curry(f)\circ\pi_1)$.  By
  \eqref{eq:D-curry-1}, the left hand side is equal to
  $\uncurry(D(\curry(f))) = \uncurry(\curry(D(f)\circ t')) = D(f)\circ
  t'$, while the right hand side is equal to
  $\uncurry(\curry(f)\circ\pi_1) = \uncurry(\curry(f))\circ
  (\pi_1\times\id) = f\circ (\pi_1\times\id)$ by the definition of
  $\uncurry$, hence the assertion.
\end{proof}

\begin{theorem}\label{thm-und-T-linear}
  $\underline{T}$ is a linear morphism.
\end{theorem}

\begin{proof}
  $\underline{T}$ is the currying of the morphism $T(\ev)\circ t =
  \<D(\ev)\circ t, \ev\circ(\id\times\pi_2)\>$.  Let us check that the
  condition of \lemref{lemma-curry-linear} is satisfied.  By D4,
  D5, and \eqref{eq:circ-pair}, we have:
  \begin{equation}
    \label{eq:D-T-ev-t}
    D(T(\ev)\circ t)\circ t' = \<D(D(\ev))\circ T(t)\circ t',
    D(\ev)\circ T(\id\times\pi_2)\circ t'\>.
  \end{equation}
  Differentiating equation~\eqref{eq:D-uncurry-2}, we obtain
  $D(D(\ev)\circ t') = D(\ev\circ(\pi_1\times\id))$.  By D5, the left
  hand side is equal to $D(D(\ev))\circ T(t')$ and the right hand side
  is equal to $D(\ev)\circ T(\pi_1\times\id)$.  Precomposing both sides
  of the equation with $t$, we conclude that
  \begin{alignat*}{3}
    D(D(\ev))\circ T(t')\circ t
    &= D(\ev)\circ T(\pi_1\times\id)\circ t
    \\
    &= D(\ev)\circ t \circ (\pi_1\times T(\id))
    & \qquad & \textup{by naturality of $t$}
    \\
    &= D(\ev)\circ t\circ (\pi_1\times\id)
    & \qquad & \textup{by functoriality of $T$}.
  \end{alignat*}
  By \proref{prop-sigma-compat-strength}, $T(t')\circ t =
  \sigma\circ T(t)\circ t'$, therefore the left hand side of the above
  equation is equal to $D(D(\ev))\circ\sigma\circ T(t)\circ t'$, which
  is equal to $D(D(\ev))\circ T(t)\circ t'$ by
  \corref{cor-D-interchange}.  We conclude that
  \[
  D(D(\ev))\circ T(t)\circ t' = D(\ev)\circ t\circ (\pi_1\times\id).
  \]
  Furthermore:
  \begin{alignat*}{3}
    D(\ev)\circ T(\id\times\pi_2)\circ t'
    &= D(\ev)\circ t'\circ (\id\times\pi_2)
    & \qquad & \textup{by naturality of $t'$}
    \\
    &= \ev\circ(\pi_1\times\id)\circ(\id\times\pi_2)
    & \qquad & \textup{by \eqref{eq:D-uncurry-2}}
    \\
    &= \ev\circ(\id\times\pi_2)\circ(\pi_1\times\id)
    & \qquad & \textup{by functoriality of $\times$}.
  \end{alignat*}
  Plugging these expressions into \eqref{eq:D-T-ev-t}, we obtain
  \begin{alignat*}{3}
    D(T(\ev)\circ t)\circ t'
    &= \<D(\ev)\circ t\circ (\pi_1\times\id),
    \ev\circ(\id\times\pi_2)\circ(\pi_1\times\id)\>
    \\
    &= \<D(\ev)\circ t, \ev\circ(\id\times\pi_2)\> \circ
    (\pi_1\times\id)
    & \qquad & \textup{by \eqref{eq:circ-pair}}
    \\
    &= T(\ev)\circ t\circ (\pi_1\times\id)
    & \qquad & \textup{by \eqref{eq:T-ev-t}}.
  \end{alignat*}
  Applying \lemref{lemma-curry-linear}, we conclude that
  $\underline{T}$ is a linear morphism.
\end{proof}

\begin{proposition}\label{prop-curry-T-t}
  Let $f:Z\times X\to Y$ be a morphism in $\C$.  Then
  $\underline{T}\circ\curry(f) = \curry (T(f)\circ t)$.
\end{proposition}

\begin{proof}
  Equivalently, $\uncurry(\underline{T}\circ\curry(f)) = T(f)\circ
  t$.  We have:
  \begin{alignat*}{3}
    \uncurry(\underline{T}\circ\curry(f))
    & = \ev\circ (\underline{T}\circ\curry(f)\times\id)
    & \qquad & \textup{by def. of $\uncurry$}
    \displaybreak[0]
    \\
    &= \ev\circ (\underline{T}\times\id)\circ (\curry(f)\times\id)
    & \qquad & \textup{by functoriality of $\times$}
    \displaybreak[0]
    \\
    &= T(\ev)\circ t\circ (\curry(f)\times\id)
    & \qquad & \textup{by def. of $\underline{T}$}
    \displaybreak[0]
    \\
    &= T(\ev) \circ T(\curry(f)\times\id)\circ t
    & \qquad & \textup{by naturality of $t$}
    \displaybreak[0]
    \\
    &= T(\ev\circ(\curry(f)\times\id))\circ t
    & \qquad & \textup{by functoriality of $T$}
    \displaybreak[0]
    \\
    &= T(f)\circ t
    & \qquad & \textup{by def. of $\curry$}.
  \end{alignat*}
  The proposition is proven.
\end{proof}

\section{Conclusions}

In this note we have introduced the notion of tangent bundle in any
cartesian differential category $\C$.  We have shown that the tangent
bundle functor $T$ is part of a strong commutative monad.  In
particular, when the category $\C$ is cartesian closed, the general
theory of strong monads has allowed us to conclude that the functor
$T$ is closed and admits an enrichment.  We have computed these
structures more explicitly when $\C$ is a differential
$\lambda$-category.

\bibliography{acm,mathscinet,tan-bun-in-diff-cats}
\bibliographystyle{plainnat}

\end{document}